\documentclass[oneside,english, 11pt]{amsart}
\usepackage[T1]{fontenc}
\usepackage[utf8]{inputenc}
\usepackage{xcolor}
\usepackage{amsthm}
\usepackage{amstext}
\usepackage{amssymb}
\usepackage{amsmath}
\usepackage{mathtools}
\usepackage[mathscr]{eucal}
\usepackage{algorithm}
\makeatletter
\numberwithin{equation}{section}
\numberwithin{figure}{section}
\theoremstyle{plain}
\newtheorem{thm}{\protect\theoremname}
  \theoremstyle{plain}
  \newtheorem{pro}[thm]{\protect\propositionname}
   \theoremstyle{plain}
  \newtheorem{cor}[thm]{\protect\corollaryname}
  \theoremstyle{remark}
  \newtheorem{rem}[thm]{\protect\remarkname}
  \theoremstyle{definition}
  \newtheorem{defn}[thm]{\protect\definitionname}
  \theoremstyle{plain}
  \newtheorem{lem}[thm]{\protect\lemmaname}
    \theoremstyle{definition}

\theoremstyle{remark}

\usepackage[english]{babel}
  \providecommand{\definitionname}{Definition}
  \providecommand{\lemmaname}{Lemma}
  \providecommand{\propositionname}{Proposition}
   \providecommand{\corollaryname}{Corollary}
  \providecommand{\remarkname}{Remark}
\providecommand{\theoremname}{Theorem}
\providecommand{\examplename}{Example}
\providecommand{\assertionname}{Assertion}

\DeclareMathOperator{\re}{Re}
\DeclareMathOperator{\im}{Im}

\DeclareMathOperator{\id}{id}
\DeclareMathOperator{\diff}{Diff}
\DeclareMathOperator{\sing}{Sing}

\DeclareMathOperator{\gl}{GL}

\DeclareMathOperator{\fol}{Fol}
\DeclareMathOperator{\hol}{Hol}



\newboolean{javier}
\setboolean{javier}{true}
\newcommand{\javier}{\ifthenelse{\boolean{javier}}{\color{red}
    \setboolean{javier}{false}}{\color{black}\setboolean{javier}{true}}}    
    
\newboolean{rudy}
\setboolean{rudy}{true}
\newcommand{\rudy}{\ifthenelse{\boolean{rudy}}{\color{blue}
    \setboolean{rudy}{false}}{\color{black}\setboolean{rudy}{true}}}

\author{Javier Rib\'{o}n}
\author{Rudy Rosas}

\makeatletter
\@namedef{subjclassname@2020}{%
  \textup{2020} Mathematics Subject Classification}
\makeatother

\begin{document}

\title[Holomorphic foliations with prescribed holonomy]
{Nondegenerate germs of  holomorphic foliations with prescribed holonomy}

\maketitle

\section*{Abstract}

We are interested in characterizing the holonomy maps associated to integral curves of non-degenerate 
singularities of holomorphic vector fields. 
Such a description is well-known in dimension 2 where is a key ingredient in the study of reduced singularities.
The most intricate case in the 2 dimensional setting 
corresponds to  (Siegel) saddle singularities. This work treats the analogous problem for 
saddles in higher dimension.

We show that any germ of holomorphic biholomorphism, in any dimension,  
can be obtained as the holonomy map associated to an  integral curve of a saddle singularity. 

A natural question is whether we can prescribe the linear part of the saddle germ 
of vector field provided the holonomy map.
The answer to this question is known to be positive in dimension 2. We see that this is not the case in higher 
dimension. In spite of this, we provide a positive result under a natural condition for the holonomy map.

\section{Introduction}

We consider a germ of singular holomorphic foliation  in $(\mathbb{C}^{n+1},0)$ defined by  a system of the form 

\begin{align}\label{sistema0}
\begin{aligned}
x'&= x\\
y'&= A\cdot y+ G(x,y),
\end{aligned}
\end{align}
where $y=(y_1,\dots,y_n)$,  the function $G\colon (\mathbb{C}^{n+1},0)\to (\mathbb{C}^{n},0)$ is holomorphic with $G(x,0)=0$, 
 $dG (0) =0$  and $A\in\gl(n,\mathbb{C})$.
As it can be easily checked, the singularity 
at the origin is isolated and the complex manifolds $\{x=0\}$ and $\{y=0\}$ are invariant by the foliation.
We denote by $\fol ({n+1})$ the set of germs of foliations as above. 
Given $\mathcal{F}\in \fol(n+1)$, the invariant  manifold $S\colon =\{y=0\}$ has complex dimension one and $S\backslash \{0\}$ is a regular leaf
of $\mathcal{F}$. Thus, we can consider the holonomy associated to a small positively oriented simple loop around the origin in $S\backslash \{0\}$. Up to holomorphic conjugation, this holonomy is a germ
of biholomorphism in $\diff (\mathbb{C}^{n},0)$ of the form 
$$y\mapsto  e^{2\pi i A}\cdot y+ \textrm{h.o.t.}$$  
We denote this map by $\hol(\mathcal{F})$ and call it the holonomy of $\mathcal{F}$ or the holonomy of the system \eqref{sistema0}. In this way we have  the map
  \begin{align}\label{fola}\hol\colon \fol(n+1)\to \diff(\mathbb{C}^{n},0).
 \end{align}
Note that if two foliations ${\mathcal F}$ and ${\mathcal F}'$ in  $\fol(n+1)$ are conjugated by a germ of biholomorphism defined in a neighborhood of the origin and 
preserving $S$, then their holonomies $\hol(\mathcal{F})$ and $\hol(\mathcal{F}')$ are analytically conjugated.
So $\hol$ induces a map $[\hol]$ between the moduli spaces of $\fol(n+1)$  and $\diff(\mathbb{C}^{n},0)$ for the analytic classification.
As a consequence, it is vital to study the properties of the map $[\hol]$ to describe the analytic moduli of foliations in $\fol(n+1)$.
In this paper we show the following theorem.
\begin{thm}\label{maintro}The map $[\hol]$ is surjective.
\end{thm}
If $n=1$,  this result is consequence of the works of Martinet and Ramis in  \cite{MaRa} and P\'{e}rez-Marco and 
Yoccoz  in \cite{PeYo}.  Precisely, from these works we extract the following theorem.
\begin{thm}\label{pema}Let $h(z)=bz+O(z^2)$ be a diffeomorphism in $\diff(\mathbb{C},0)$ with $|b|=1$. Then,
given $a<0$ with $e^{2\pi i a}=b$, the diffeomorphism $h$ is --- up to holomorphic conjugation --- 
the holonomy of  a system of the form 
 \begin{align}\nonumber
\begin{aligned}
x'&= x\\
y'&= a (y+ \cdots).
\end{aligned} 
\end{align}
\end{thm}
If $|b|\neq 1$, it is well known that a diffeomorphism of the form $h(z)=bz+O(z^2)$  is always linearizable, so  $h$ is obtained as holonomy of the linear system  
 \begin{align}\nonumber
\begin{aligned}
x'&= x\\
y'&= ay
\end{aligned} 
\end{align}
for any $a\in\mathbb{C}$ with $e^{2\pi i a}=b$. \\ 
In order to make a more precise statement of Theorem \ref{maintro}, let us introduce some definitions.  
We say that $A\in \gl(n,\mathbb{C})$ has {\it negative spectrum} if all eigenvalues of $A$ have negative real part. 
\begin{rem}
Consider a system of the form \eqref{sistema0} where $A$ has negative spectrum, $G(0)=0$ and $dG (0) =0$. 
Since the imaginary line divides 
the eigenvalues of the system in two groups, namely the singleton $\{1\}$ and the set of eigenvalues of $A$, there exists a holomorphic invariant manifold 
$S$, of dimension $1$, tangent to $\{y=0\}$ at the origin by Hadamard-Perron theorem (see \cite[Theorem 7.1]{Ilya-Yako}). 
Hence, we can assume that the system is in $\fol(n+1)$ up to a holomorphic change of coordinates. So, in the negative spectrum case, 
it suffices to work in $\fol(n+1)$ to study the analytic moduli.
\end{rem}
Let us remark that for $n=1$ and $A$ of negative spectrum, the map $[\hol]$ is a bijection between moduli spaces \cite{MaMo},  see \cite[Theorem 22.7]{Ilya-Yako}.
Thus, the analytic classification of foliations is reduced to an analogous problem for discrete maps.
It is not known in general whether $[\hol]$ is injective for $n \geq 2$ even if there are some partial results in that direction \cite{Eli-Ilya}, see also \cite{Reis}.
 \begin{defn}
 \label{defn:negative}
 Let $A$ be an $n \times n$ complex matrix.  We can assume that, up to a linear change of coordinates, its 
 semisimple part $A_s$ of the additive Jordan decomposition (cf. Definition \ref{def:jordan})   
 is of the form $A_s = \mathrm{diag} (\mu_1, \hdots, \mu_n)$. 
  Given $j= (j_1, \hdots, j_n) \in {\mathbb Z}_{\geq 0}^{n}$ and $1 \leq k \leq n$ we define
 \[ R_{j;k} =  R_{j_1 \hdots j_n; k}:= \mu_k  - j_1 \mu_{1} - \hdots - j_n \mu_{n} . \]
 Given $(j;k)$ such that $j \in {\mathbb Z}_{\geq 0}^{n}$, $|j|:=j_1 + \hdots +j_n \geq 2$ and
 $1 \leq k \leq n$, we say that it is a {\it positive resonance} (resp. negative, non-positive, non-negative) 
 of $e^{2 \pi i A}$ (with respect to $A$) 
 if $R_{j;k}$ is a positive (resp. negative, non-positive, non-negative)  integer number. 
 \end{defn}

 We prove the following theorem and its immediate corollary. 
\begin{thm}\label{maintro1}Consider  $h(y)= e^{2 \pi i A}  \cdot y+\textrm{h.o.t.}$ in $\diff(\mathbb{C}^n,0)$, where
$A\in \gl(n,\mathbb{C})$ has negative spectrum. 
Suppose that there exists a formal normal
 form $\tilde{h}$ of $h$ with respect to $A$ 
with no negative resonances (of $e^{2 \pi i A}$). Then, 
the diffeomorphism $h$ is --- up to holomorphic conjugation --- 
the holonomy of  a system of the form \eqref{sistema0}. 
\end{thm}
 \begin{cor}\label{maincor}
Consider  $h(y)= e^{2 \pi i A} \cdot y+\textrm{h.o.t.}$ in $\diff(\mathbb{C}^n,0)$, where
 $A\in \gl(n,\mathbb{C})$ has negative spectrum.
Suppose that $e^{2 \pi i A}$ has no negative resonances. 
Then, the diffeomorphism $h$ is --- up to holomorphic conjugation --- 
the holonomy of  a system of the form \eqref{sistema0}. 
\end{cor}
The concepts of formal normal form and negative resonances are introduced in Definitions \ref{def:normal} and \ref{def:resonance}.
In order to illustrate the situation,   assume 
that $A=A_s = \mathrm{diag} (\mu_1, \hdots, \mu_n)$ for the sake of simplicity, the general case is treated later on.
Let us remark that $\tilde{h}$  is in formal normal form with respect to $A$ if it has a Taylor power series expansion at the origin of the form
\[  \tilde{h} (y_1, \hdots, y_n) = 
 \left( \sum_{|j| \geq 1} a_{j;  1} \,  y^{j}, \hdots, \sum_{|j| \geq 1 } a_{j;  n} \,  
 y^{j} \right) , \]
   where any coefficient $a_{j;k}$ such that $R_{j;k} \not \in \mathbb{Z}$ vanishes. 
 The diffeomorphism $h$ has always a formal normal form $\tilde{h}$, i.e. $h$ is  conjugated to some $\tilde{h}$ in formal normal form
 by a formal diffeomorphism with identity linear part. Note that the formal normal form is not unique in general. 
  The condition of absence of negative resonances amounts to requiring that 
 whenever $a_{j;k} \neq 0$, we have $R_{j;k} \geq 0$.
 
 The scope of application of Theorem \ref{maintro1} is very broad. For instance,  fixed $A$, with negative spectrum, 
 there exist $\nu \in {\mathbb N}$ such that any negative resonance $(j,k)$ has degree $|j|   \leq \nu$.
 Therefore, we can apply Theorem \ref{maintro1} if $h$ is linearizable up to order $\nu$ and in particular if it is formally linearizable.


Theorem \ref{pema}  is a direct consequence of
Corollary \ref{maincor} since in dimension $n=1$ negative spectrum implies no negative resonances. 
The necessity of the condition on negative resonances  will be explained in Section \ref{secexample}.

Let us show that Theorem \ref{maintro} is a direct consequence of Theorem \ref{maintro1}. To do so,  
it is enough to show that,
 given 
 $B\in\gl(n,\mathbb{C})$, there exists $A\in\gl(n,\mathbb{C})$ 
 such that $e^{2\pi i A}=B$ and $B$ has no negative resonances with respect to $A$.
 Firstly, we take any $n\times n$ complex matrix $A$ such that
$\nonumber e^{2\pi i A}=B$.
For each $m\in\mathbb{N}$, set 
$A_m\colon = A- mI,$ where $I$ is the identity matrix. Since $A$ and $m I$ commutes, we have 
\begin{align}\nonumber e^{2\pi i A_m}=e^{2\pi i A}e^{-2\pi i m I}=B. 
\end{align} Moreover, it is clear that $A_m$ has negative spectrum and
$B$ has no negative resonances with respect to $A_m$ if $m>>1$.

   This work is organized as follows. 
 In section \ref{resonances} we introduce some known properties of difeomorphisms and normal forms and study
 the concept of negative resonances.
 In section \ref{secsingular} we state and prove an extension property about
 singular  holomorphic foliations.  In section \ref{seccons} we solve the realization of the holonomy problem up to any 
 finite order by building model foliations for diffeomorphisms with normal forms without negative resonances. 
 In section \ref{secmain} we outline the proof of Theorem
 \ref{maintro1} and reduce it to the proof of the propositions \ref{smoothrealization}, \ref{jexten} and \ref{teonw}. In Section \ref{secspecial} we study some class of functions defined on angular sectors, we define a notion of order at the origin for these functions. Finally, 
 Proposition \ref{smoothrealization} is proved in  section \ref{secsmooth}, and Propositions  \ref{jexten} and \ref{teonw} in Section \ref{secextension}.
 We provide an application of our main result in section \ref{application}.

\section{Properties of local diffeomorphisms and resonances} \label{resonances} 
Let $\diff (\mathbb{C}^n,0)$ be the group of germs of biholomorphisms $f: U \to V$  with $f(0)=0$ where $U$ and $V$ are open neighborhoods of $0$
in ${\mathbb C}^{n}$. Its formal completion $\widehat{\diff} (\mathbb{C}^n,0)$, the so
called group of formal diffeomorphisms,  
consists of the Taylor power series expansions of the form 
 \[ f (y_1, \hdots, y_n) = 
 \left( \sum_{|j| \geq 1  } a_{j;  1} \,  y^{j}, \hdots, \sum_{|j| \geq 1  } a_{j;  n} \,  
 y^{j} \right)  \]
 whose linear part at the origin
 \[ df (0) (y_1, \hdots, y_n) = 
 \left( \sum_{|j| = 1  } a_{j;  1} \,  y^{j}, \hdots, \sum_{|j| = 1  } a_{j;  n} \,  
 y^{j} \right)  \]
 belongs to $\gl(n,\mathbb{C})$. For  any $k \in {\mathbb N}$, we have that
 \begin{itemize}
 \item the  $k$th-jet $j^{k} (f \circ g)$ depends on $j^{k} f$ and $j^{k} g$ for $f, g \in  \diff (\mathbb{C}^n,0)$ and
 \item the  $k$th-jet $j^{k} (f^{-1})$ depends on $j^{k} f$   for $f \in  \diff (\mathbb{C}^n,0)$.
 \end{itemize}
 As a consequence, the composition defines a group operation in $\widehat{\diff} (\mathbb{C}^n,0)$ that makes 
$\diff (\mathbb{C}^n,0)$ a subgroup of  $\widehat{\diff} (\mathbb{C}^n,0)$.
 
 \subsection{Jordan-Chevalley decomposition and normal forms}
 \label{subsec:jordan}
 Given $f\in\diff (\mathbb{C}^n,0)$, 
 it can be expressed uniquely in the form  $f=f_{s} \circ f_{u}$ where 
 \begin{itemize}
 \item $f_s, f_u \in \widehat{\diff} (\mathbb{C}^n,0)$;
 \item $f = f_s \circ f_u = f_u \circ f_s$;
 \item $f_s$ is formally conjugated to a diagonal linear map and 
 \item $f_u$ is unipotent, i.e. $\mathrm{spec} (df_u (0)) = \{1\}$. 
 \end{itemize}
  This is the analogue for local diffeomorphisms of 
 the multiplicative Jordan decomposition for linear operators
 \cite{MarJ}, see also \cite{JR:finite}.  
 We say that $f$ is in {\it formal normal form} if $f_s$ is a diagonal linear map. 
 We say that $\tilde{f} \in \widehat{\diff} (\mathbb{C}^n,0)$ is a formal normal form of $f$ if $\tilde{f}$ is in formal normal form and 
 $f$ is formally conjugated to $\tilde{f}$.
 
 \strut
 
 Fix $A\in\gl(n,\mathbb{C})$ such that $e^{2\pi i A} = d f (0)$. In the realization problem, we want to obtain $f$
 as the holonomy of a one dimensional vector field of the form \eqref{sistema0}, so we need to study 
 simultaneously the resonances of both the vector field and $f$. Therefore it is convenient to introduce 
 a concept of formal normal form for $f$ in which the role of $A$ is considered.
 \begin{defn}
 \label{def:jordan}
 Given an $n \times n$ complex matrix $M$, let 
 $M = M_s + M_N$ be the additive Jordan decomposition of $M$ as a sum of 
 linear operators that commute such that $M_s$ is diagonalizable and $M_N$ is nilpotent.
 \end{defn}
 \begin{defn}
 \label{def:normal}
 We say that $\tilde{f}$ is a {\it formal normal form} of $f$ with respect to $A$ if there exists $\psi \in \widehat{\diff} (\mathbb{C}^n,0)$
 such that 
 \begin{itemize}
 \item $\psi$ conjugates $f$ to $\tilde{f}$ and $d \psi (0)$ conjugates $A$ to $\tilde{A}$,
 \item $\tilde{A}_s$ is a diagonal matrix $\mathrm{diag} (\mu_1, \hdots, \mu_n) $ and
 \item  $\tilde{f}_s$ is of the form $\tilde{f}_{s} (y_1, \hdots, y_n) = (\lambda_1 y_1, \hdots, \lambda_n y_n)$ with $\lambda_j = e^{2 \pi i \mu_j}$.
 \end{itemize}
\end{defn}
  A formal normal form can be obtained  by considering a linear conjugacy that
 diagonalizes $A_s$ and then conjugating $f_s$ to $(\lambda_1 y_1, \hdots, \lambda_n y_n)$ by a tangent to the identity formal
 diffeomorphism.
 Notice that a formal normal form of $f$ with respect to $A$ is not necessarily unique.

  \subsection{Resonances}
  We can write $\tilde{f}$ in the form
 \[ \tilde{f} (y_1, \hdots, y_n) = \left( \sum_{|j| \geq 1} a_{j;  1} \,  y^{j}, \hdots, \sum_{|j| \geq 1 } a_{j;  n} \,  y^{j} \right) =  
 \sum_{|j| \geq 1, \ , 1 \leq k \leq n} a_{j;  k}  \, y^{j} \, e_k\]
 where $j$ runs on multi-indices $(j_1, \hdots, j_n) \in {\mathbb Z}_{\geq 0}^{n}$ with 
 $|j|:= j_1 + \hdots +j_n \geq 1$,  $y^{j} =y_{1}^{j_1} \hdots y_{n}^{j_n}$ and 
$e_k$ is the $k$th-element of the canonical basis of ${\mathbb C}^{n}$.
\begin{defn}
\label{def:resonance}
We say that $(j;k)$ is a resonance of $\tilde{f}$ if $a_{j;k} \neq 0$.
\end{defn}
Consider $\tilde{f}$ and $\tilde{A}$ as in Definition \ref{def:normal}. 
  If $a_{j_1 j_2 \hdots j_n;  k} \neq 0$, since $f_s\circ f_u=f_u\circ f_s$, we have that 
  $$(y_{1}^{j_1} \hdots y_{n}^{j_n} e_k)  \circ \tilde{f}_s \equiv \tilde{f}_s \circ( y_{1}^{j_1} \hdots y_{n}^{j_n} e_k),$$
 which is equivalent to the equality $ \lambda_k \lambda_{1}^{-j_1} \hdots \lambda_{n}^{-j_n} =1$. Then, since 
 $$ e^{2 \pi i  R_{j_1 \hdots j_n; k}}  =  e^{2 \pi i  (\mu_k  - j_1 \mu_{1} - \hdots - j_n \mu_{n} )} = 
   \lambda_k \lambda_{1}^{-j_1} \hdots \lambda_{n}^{-j_n} ,$$ $a_{j_1 j_2 \hdots j_n;  k} \neq 0$ implies 
  $ R_{j_1 \hdots j_n; k} \in {\mathbb Z}$.
  By construction, all resonances $(j;k)$ of $\tilde{f}$ with $|j|=1$ satisfy $R_{j;k}=0$.  
 The signs of the resonances $ R_{j; k}$ play a significant role in the realization Theorem \ref{maintro1}.
 \begin{rem}
 \label{rem:jordan}
 If necessary, we can assume that given a formal normal form $\tilde{f}$ of $f$ with respect to $A$, the matrix 
 $\tilde{A}$ is in Jordan normal form, preserving the signs of resonances. More precisely, there is a  linear map 
 $L$ that conjugates $\tilde{A}$ to a matrix in Jordan normal form and commutes with $\tilde{A}_s$.   Notice that 
 \[ \left[  y_{1}^{j_1} \hdots y_{n}^{j_n} \frac{\partial}{\partial y_k} , \mu_1 y_1   \frac{\partial}{\partial y_1} + \hdots +  \mu_n y_n   \frac{\partial}{\partial y_n} 
 \right]  = R_{j;k}    y_{1}^{j_1} \hdots y_{n}^{j_n} \frac{\partial}{\partial y_k}  \]
 and denote  $W= L_{*} (y_{1}^{j_1} \hdots y_{n}^{j_n} \frac{\partial}{\partial y_k} )$. Then, since $L \tilde{A}_s = \tilde{A}_s L$, 
  \[ \left[ W , \mu_1 y_1   \frac{\partial}{\partial y_1} + \hdots +  \mu_n y_n   \frac{\partial}{\partial y_n} 
 \right]  = [W,\tilde{A}_s] =  R_{j;k}  W . \]
 Therefore, the resonances associated to the non-vanishing monomials of 
  \[  L \circ  y^{j} e_k \circ L^{-1}  \quad \mathrm{or} \quad L_{*} \left( y_{1}^{j_1} \hdots y_{n}^{j_n} \frac{\partial}{\partial y_k} \right) \]
 are all equal to $R_{j;k}$.  Thus, 
 if $(j;k)$ is a positive resonance (resp, negative, non-positive, non-negative) 
 then all non-vanishing monomials of  $ L \circ  y^{j} e_k \circ L^{-1}$
 are positive (resp. negative, non-positive, non-negative). In particular, $\tilde{f}$ has negative resonances if and only if 
 $ L \circ  \tilde{f} \circ L^{-1}$ does.
  \end{rem}

\subsection{The exponential map}
  The presentation in this subsection is classical and it is included for the sake of completeness. More details can be found in \cite{JR:finite}.
  We say that $X$ is a formal vector field  in $(\mathbb{C}^n,0)$ if it is a derivation of the ring ${\mathbb C}[[y_1, \hdots, y_n]]$
  that preserves its maximal ideal. It can be expressed in the form
  \[ b_1 (y_1,\hdots, y_n) \frac{\partial}{\partial y_1} + \hdots +  b_n (y_1,\hdots, y_n) \frac{\partial}{\partial y_n} \]
  where $b_j = X(y_j)$ for $1 \leq j \leq n$. 
  We say that $X$ is nilpotent if its first jet 
  $j^{1} X = \sum_{j=1}^{n} j^{1} b_j  \; \partial/\partial_{j}$ is a nilpotent linear operator.
  We denote by $X^{k} (f)$ the result of applying $k$ times the operator $X$ to $f$.

  We can define the exponential $\mathrm{exp} (X)$ of a formal vector field. 
  Given an analytic vector field $X$, i.e. a formal vector field such that $X$ preserves  
  ${\mathbb C}\{y_1, \hdots, y_n\}$,  we consider its flow $\phi (t,y)= (\phi_1 (t,y), \hdots, \phi_n (t,y))$. It satisfies  
   
  \[ \frac{\partial^{k} \phi_j (t,y)}{\partial t^{k}} (0,y) =  X^{k} (y_j)  (y)  \]
in a neighborhood of $0$ in ${\mathbb C}^{n}$ 
for all $1 \le j \le n$ and $k \geq 1$. The result is trivial if $y$ is a singular point of $X$ and derives of Taylor's formula if $y$ is regular
since every regular vector field is locally equivalent to a coordinate vector field. Thus, we get 
\[ \phi_j (t,y) = y_j + \sum_{k=1}^{\infty} t^{k} \frac{ X^{k} (y_j)}{k!}  \]
for any $1 \leq j \leq n$. By definition, the exponential $\mathrm{exp}(X)$ of $X$ is the time $1$ map $\phi(1,y)$. 

The definition of $\mathrm{exp}$ generalizes to formal vector fields. Indeed given a formal vector field $X$ in $(\mathbb{C}^n,0)$, 
its exponential $\mathrm{exp} (X)$ is the formal diffeomorphism 
  $f = (f_1, \hdots, f_n) \in \widehat{\diff} (\mathbb{C}^n,0)$ such that
  \begin{equation}
  \label{taylor}
   f_j = y_j + \sum_{k=1}^{\infty} \frac{X^{k} (y_j)}{k!}  
  \end{equation}
  for any $1 \le j \le n$. The next result is classical. 
  \begin{pro}[cf, \cite{MaRa, Ecalle}  {\cite[Th. 3.17]{Ilya-Yako}} ] 
  \label{pro:exp}
  The exponential map establishes a bijection between formal nilpotent vector fields and formal unipotent diffeomorphisms. 
  \end{pro}
 \begin{rem}
 \label{rem:log}
 Given a formal unipotent diffeomorphism $f$, the unique formal nilpotent vector field whose exponential is $f$ is called its {\it infinitesimal generator}. 
 There is a formula to calculate the infinitesimal generator $\log f$ of $f$. Let $\Theta$ be the linear operator defined in ${\mathbb C}[[y_1, \hdots, y_n]]$
 by $\Theta (h) = h \circ f - h$. We have
 \begin{equation}
 \label{equ:log}
  (\log f)(y_j) = \sum_{k=1}^{\infty} (-1)^{k+1} \frac{\Theta^{k} (y_j)}{k}  
  \end{equation}
 for any $1 \leq j \leq n$.
 This formula is obtained by noticing that the operator $h \mapsto h \circ f$ defined by $f$ in ${\mathbb C}[[y_1, \hdots, y_n]]$ is of the form $\mathrm{id} + \Theta$
 and considering the development of $\log (1 + z)$. 
  \end{rem}
   
\section{Singular holomorphic foliations}\label{secsingular}

 Let $M$ be a complex manifold of dimension $m\ge2$. A singular holomorphic foliation by curves 
 $\mathcal{F}$ on 
 $M$ can be defined by a collection $\{(U_i,Z_i)\}_{i\in I}$, where  $\{U_i\}_{i\in I}$ is an open covering of $M$ and each  $Z_i$ is a holomorphic vector field on $U_i$ with singular set of codimension $\ge 2$,   such that on each $U_i\cap U_j$  the vector fields $Z_i$ and $Z_j$ coincides up to multiplication by a nowhere vanishing holomorphic function. A point $p\in M$ is a singularity of $\mathcal{F}$ if it is a singularity  for some  $Z_i$. Thus, the set $\sing (\mathcal{F})$ of the singularities of $\mathcal{F}$  is an analytic variety of codimension $\ge 2$. The restriction of $\mathcal{F}$ to $M\backslash \sing(\mathcal{F})$ is a holomorphic foliation of complex dimension one,  in the classical sense. Conversely,  if $V\subset M$ is an analytic variety of codimension
  $\ge 2$ and $\mathcal{F}$ is a  holomorphic foliation of complex dimension one on $M\backslash V$, it
  is well known that $\mathcal{F}$ extends to $M$ as a singular  holomorphic foliation by curves. 
  If $V\subset M$ is any proper subvariety, we have the following extension criteria.
 
\begin{pro}\label{extefol} Let $M$ be a complex manifold of dimension $m\ge2$, let $V\subset M$ be a proper subvariety, and  
let  $\mathcal{F}$ be a holomorphic foliation of complex dimension one on $M\backslash V$. Suppose that $\mathcal{F}$ extends to $M$ as a $C^1$ foliation $\bar{\mathcal{F}}$. Then $\bar{\mathcal{F}}$ is holomorphic. 
\end{pro}
\proof
  Let $p\in V$ be arbitrary. 
 If $U$ is a small neighborhood of $p$ in $M$, we can directly assume $p\in U\subset \mathbb{C}^m$.
 Given $x\in U$, the tangent line $T_x\bar{\mathcal{F}}$ of $\bar{\mathcal{F}}$ at $x$ is a  
 real bidimensional subspace of $\mathbb{C}^m$. Since $T_x\bar{\mathcal{F}}$ depends continuously
 on $x\in U$  and $T_x\bar{\mathcal{F}}$ 
 is a complex line whenever $x\in U\backslash V$, we deduce that $T_x\bar{\mathcal{F}}$ is a complex line for all $x\in U$. Thus,  we can suppose
 $$T_p\bar{\mathcal{F}}=[1:0:\cdots : 0]\in \mathbb{P}^{m-1}$$ and, by reducing $U$ if necessary,
  we have  $T_x\bar{\mathcal{F}}$, $x\in U$ so close to $[1:0:\cdots : 0]$ such that we can write 
 $$T_x\bar{\mathcal{F}}=[1:\theta_1(x):\cdots :\theta_{m-1}(x)],$$
where $$\theta_1,\dots,\theta_{m-1}\colon U\to\mathbb{C}$$ are continuous.
Thus, since the functions $\theta_j$ need to be holomorphic on $U\backslash V$ 
--- because $\bar{\mathcal{F}}|_{U\backslash V}$ is holomorphic --- we have that
 $\theta_1,\dots,\theta_{m-1}$ are indeed holomorphic on $U$. Therefore 
 $\bar{\mathcal{F}}|_{U}$ is generated by the holomorphic vector field 
 $$x\mapsto (1,\theta_1(x),\dots ,\theta_{m-1}(x)),$$
 so $\bar{\mathcal{F}}$ is holomorphic and
 regular around $p\in V$. 
\qed

We are mainly interested in foliations defined by holomorphic vector fields around an isolated singularity,
so we consider a singular holomorphic foliation $\mathcal{F}$ defined by a vector field  $Z$ on a neighborhood $U$ of $0\in\mathbb{C}^m$ with a unique singularity at the origin. An analytic curve $S\subset U$ through the origin is said to be a separatrix of $\mathcal{F}$
if $S$ is tangent to $Z$. In this case $S\backslash\{0\}$ is contained in a leaf of the regular foliation
given by the restriction of $\mathcal{F}$ to
$U\backslash\{0\}$. Thus, by definition, the holonomy of $S$ is the holonomy map associated to any  positively oriented simple loop around the origin in $S\backslash\{0\}$. Up to holomorphic conjugation, this holonomy map is a germ in $\diff (\mathbb{C}^{m-1},0)$.

\section{Negative resonances as an obstruction to realization}\label{secexample}

In this section, we present an example $h(y)=e^{2 \pi i A} \cdot y+\textrm{h.o.t.}$ in $\diff(\mathbb{C}^2,0)$
such that $h$ is not the holonomy map of any system of the form (\ref{sistema0}). This phenomenon will be due to the presence
of negative resonances of $dh(0)$.

We take
\[ A = \left( 
\begin{array}{rr}
-\frac{3}{2} & 0 \\
0 & - \frac{1}{4}
\end{array}
\right) \]
and define $h \in \diff(\mathbb{C}^2,0)$ by 
\[  h (y_1, y_2) = (-y_1 + y_{2}^{2}, -i y_2) = (-y_1, -i y_2) \circ (y_1 - y_{2}^{2},  y_2)   . \]
We have $e^{2 \pi i A} = d h(0)$. Notice that $ (-y_1, -i y_2) $ and $ (y_1 - y_{2}^{2},  y_2)$ commute and then 
$h_s = (-y_1, -i y_2)$ and $h_u = (y_1 - y_{2}^{2},  y_2)$. In particular, $h$ is in formal normal form with respect to $A$.
\begin{pro}
\label{pro:example}
There is no system of the form (\ref{sistema0}) whose holonomy is   is holomorphically conjugated to $h$.
\end{pro}
\begin{proof}
Suppose, aiming at a contradiction, that such a system exists and is given by the vector field
\[ X = x \frac{\partial}{\partial x} + \sum_{i +j \geq 1, \, k \geq0} a_{k ij}  x^{k} y_1^{i} y_2^{j}  \frac{\partial}{\partial y_1} + 
\sum_{i +j \geq 1, \, k \geq 0}  b_{k ij} x^{k} y_1^{i} y_2^{j}  \frac{\partial}{\partial y_2} . \]
Consider the fuchsian equation associated to the vector field
 \[ x \frac{\partial}{\partial x} + \sum_{i +j = 1, \, k \geq0} a_{k ij}  x^{k} y_1^{i} y_2^{j}  \frac{\partial}{\partial y_1} + 
\sum_{i +j = 1, \, k \geq 0}  b_{k ij} x^{k} y_1^{i} y_2^{j}  \frac{\partial}{\partial y_2} . \]
 Since the difference of its eigenvalues $-3/2 - (-1/4)$ is not an integer number,  it is analytically conjugated by a gauge transformation 
 $g \in \diff(\mathbb{C}^3,0)$, i.e. $x \circ g \equiv x$ and $g(x,0,0) \equiv (x,0,0)$,  to 
 \[ X_{1} := x \frac{\partial}{\partial x} - \frac{3}{2} y_1 \frac{\partial}{\partial y_1}  - \frac{1}{4} y_2 \frac{\partial}{\partial y_2} . \] 
 In particular, we can suppose that all coefficients of $X-X_1$ belong to the ideal 
 $(y_1,\hdots, y_n)^{2}$. 
 Next, we want to remove the terms of degree $2$ in $y$ by considering a change of coordinates of the form $\mathrm{exp} (W)$  with 
\[  W =  \sum_{k \geq 0, \, i +j =2} \alpha_{kij} x^{k} y_1^{i} y_2^{j}  \frac{\partial}{\partial y_1} + \sum_{k \geq 0, \, i +j =2}  \beta_{kij} x^{k} y_1^{i} y_2^{j}  \frac{\partial}{\partial y_2} . \]
Indeed, we have
\[ \mathrm{exp}(W)^{*} X - X_1 = (X - X_1)  + [W,X] + \frac{1}{2!} [W, [W,X]] + \hdots  =  \]
\[
  \sum_{i +j =2} (a_{kij}+ \alpha_{kij} R_{kij;2}) x^{k} y_1^{i} y_2^{j}  \frac{\partial}{\partial y_1} 
  + \sum_{i +j =2}  (b_{kij}+ \beta_{kij} R_{kij;3})  x^{k} y_1^{i} y_2^{j}  \frac{\partial}{\partial y_2}  + O(y^3),
\] 
where $\mu_1 =1$, $\mu_2=-3/2$ and $\mu_3=-1/4$ and $R_{kij;2}$, $R_{kij;3}$  are as introduced in Definition \ref{defn:negative}.
We just need to solve the equations 
\[ a_{kij}+ \alpha_{kij} R_{kij;2} = 0 \quad \mathrm{and} \quad b_{kij}+ \beta_{kij} R_{kij;3} = 0 \]
for all $k \geq 0$ and $i+j=2$. They have a unique solution since it can be seen directly that 
$R_{kij;2} \neq 0$ and $R_{kij;3} \neq 0$ for all $k \geq 0$ and $i+j=2$. 
Since $|R_{kij;2}|$ and $|R_{kij;3}|$  are bounded away from zero,  $W$ is an analytic vector field.

Summarizing, up to a change of coordinates we can assume that $X$ is of the form $X_1 + O(y^3)$.
The holonomy of such a system is of the form 
\[  (-y_1, -i y_2)  + O(y^3) \]
and therefore $h$ is analytically conjugated to a diffeomorphism whose second jet is equal to $(-y_1, -i y_2)$. 
We deduce that $j^{2} h_u \equiv \mathrm{id}$ but this contradicts $j^{2} h_u (y_1, y_2) = (y_1 - y_{2}^{2},  y_2) $.
\end{proof}
\begin{rem}
The eigenvalues $\lambda_1 = -1$ and $\lambda_2 = -i$ have exactly one resonance of degree $2$, corresponding to the monomial
$y_{2}^{2} e_1$. Our choice of $A$, and hence of $\mu_1=-3/2$ and $\mu_2 = -1/4$, makes this resonance negative since 
$\mu_1 - 2 \mu_2 = -1$. So the absence of negative resonances is needed in Theorem  \ref{maintro1}. 
\end{rem}
\begin{rem}
Positive resonances of $df(0)$ can be ``shared'' by vector fields of the form  (\ref{sistema0}).
Indeed, if $R_{j;k} = m \in {\mathbb Z}_{\geq 0}$, the positive resonance $y^{j} e_k$ is associated to $x^{m} y^{j} \frac{\partial}{\partial y_k}$, 
which is a resonant monomial for a vector field of the form \eqref{sistema0} since 
\[ \left[ x^{m} y^{j} \frac{\partial}{\partial y_k}, x \frac{\partial}{\partial x} + \sum_{l=1}^{n} 
\mu_l y_l \frac{\partial}{\partial y_l}  \right] =
 (R_{j;k} -m) x^{m} y^{j} \frac{\partial}{\partial y_k} = 0. \] 
This property will be fundamental to prove
Theorem \ref{maintro1}. Its absence is the obstruction to realize holonomies in the negative spectrum case.
For instance, in the example of Proposition \ref{pro:example}, since there are no positive resonances of degrees $1$ or $2$ in $y$, we can linearize 
the system up to order $2$ in the $y$ variables but since this property is shared by the holonomy map, 
it is incompatible with the presence of the negative resonance $y^{2} e_1$ of $f$.
\end{rem}

 \section{Construction of holonomies with positive resonances}\label{positiveconstruction}\label{seccons}

In this section we prove a realization result  up to an arbitrary order in the (positive) resonant case. 
 \begin{pro} \label{pro:positive_construction}
  Consider $f \in \widehat{\diff}(\mathbb{C}^n,0)$ and $A \in \gl(n,\mathbb{C})$ such that 
 $A_s = \mathrm{diag} (\mu_1, \hdots, \mu_n)$, $df (0) = e^{2 \pi i A}$ and 
 $f_s (y_1, \hdots, y_n) = (\lambda_1 y_1, \hdots, \lambda_n y_n)$. 
 Suppose that $A$ has negative spectrum and $f$ has no negative resonances.
 Then, given $\nu \in {\mathbb N}$, there exists a polynomial system of the form  \eqref{sistema0}
 whose holonomy $h$ (on $\{x=1\}$) satisfies $j^{\nu} h = j^{\nu} f$.
 \end{pro}
   \proof  Fix $\nu \in {\mathbb N}$.  
 Notice that $f$ is already in formal normal form.
 Denote 
  \[ f (y_1, \hdots, y_n) = \left( \sum_{|j| \geq 1} c_{j;  1} \,  y^{j}, \hdots, \sum_{|j| \geq 1 } c_{j;  n} \,  y^{j} \right) =  
 \sum_{|j| \geq 1, \, 1 \leq k \leq n} c_{j;  k}  \, y^{j} \, e_k. \]
 Since all resonances of $f$ are non-negative, we can 
 define $F \in  \diff(\mathbb{C}^{n+1},0)$ given by 
 \[ F(x,y_1, \hdots, y_n)= \left( x, \sum_{1 \leq |j| \leq \nu} c_{j;  1} \,  x^{R_{j;1}} y^{j}, \hdots, \sum_{1 \leq |j| \leq \nu } c_{j;  n} \, x^{R_{j;n}}  y^{j}
 \right) ,\]
 see Definition \ref{defn:negative}. Consider the linear vector field 
 \[ X:=  x \frac{\partial}{\partial x} + A \cdot y \frac{\partial}{\partial y}, \]  
 whose semisimple part is equal to 
  \[ X_s = x \frac{\partial}{\partial x} + A_s \cdot y \frac{\partial}{\partial y}  =
  x \frac{\partial}{\partial x} + \sum_{j=1}^{n} \mu_j y_j \frac{\partial}{\partial y_j}.  \]
 Notice that $j^{1} F = \mathrm{exp}(2 \pi i X)$ and 
 all non-vanishing monomials of $F$ are resonant for the vector field $X_s$, i.e. $F^{*} X_s = X_s$.  

 Here, we apply \cite[Theorem 1.4]{rib-embedding} to obtain a formal  vector field $Z$
 such that $F = \mathrm{exp}(2 \pi i Z)$ and  $j^{1} Z =  X$.   
 We describe the construction of $Z$ in \cite{rib-embedding} 
 in order to understand its properties. The multiplicative Jordan decomposition $F= F_s \circ F_u$ of $F$ is given by
 \[ F_s (x, y_1, \hdots, y_n) = (x, \lambda_1 y_1, \hdots, \lambda_n y_n) \quad \mathrm{and} \quad F_u = F_{s}^{-1} \circ F . \] 
 Let $\log F_u$ be the infinitesimal generator of $F_u$ (see Remark \ref{rem:log}).
 All non-vanishing monomials of $F_u$ are resonant with respect to $X_s$.
 This property is shared by $\log F_u$ \cite[Lemma 2.7]{rib-embedding}, i.e. 
 $[X_s, \log F_u] \equiv 0$. We define $Z = X_s + \frac{\log F_u}{2 \pi i}$ and so, since $X_s$ and $\log F_u$ commute,
 \[ F = F_s \circ F_u = \mathrm{exp}(2 \pi i X_s) \circ \mathrm{exp}(\log F_u) = \mathrm{exp} (2 \pi i X_s + \log F_u) = \mathrm{exp}(2 \pi i Z) . \]
Since $x \circ F_u \equiv x$, we get $(\log F_u) (x) \equiv 0$ by Equation \eqref{equ:log} and then $Z(x) \equiv x$.
 Denote by $(y)$ the ideal  $ (y_1, \hdots, y_n) $ of ${\mathbb C}[[x,y_1,\hdots,y_n]]$. The equality $(y) \circ F_u = (y)$ of ideals implies 
 that $(\log F_u) (y) \subset  (y) $ by Equation \eqref{equ:log}
 and thus $Z  (y)  \subset  (y)$. Now, it is straightforward to check that the properties
 $Z(x) \equiv x$, $Z  (y)  \subset  (y)$ and $[Z, X_s] \equiv 0$  imply that 
  $Z$ is  of the form
 \begin{equation*}
 \label{equ:auxvf}
  Z = x \frac{\partial}{\partial x} + \sum_{ 1 \leq |j|}  a_{j;  1} \, x^{R_{j,1}}  y^{j}  \frac{\partial}{\partial y_1} + \hdots + 
  \sum_{1 \leq |j|}  a_{j;  n} \, x^{R_{j,n}}  y^{j}  \frac{\partial}{\partial y_n}.  
  \end{equation*}
 Moreover, since 
 $ F_u - \mathrm{exp} (2 \pi i A_u) \cdot y = O(|y|^2)$, 
 we obtain 
 \[ \log F_u - 2 \pi i  A_u \cdot y \frac{\partial}{\partial y}  = O(|y|^{2})\] 
 by applying Equation \eqref{equ:log} to $F_u$ and $\mathrm{exp} (2 \pi i A_u) \cdot y$ and comparing the sums term by term. 
 Therefore $j^1Z = X$. Consider the polynomial vector field
 $$ Z_\nu = x \frac{\partial}{\partial x} + \sum_{ 1 \leq |j|\le \nu}  a_{j;  1} \, x^{R_{j,1}}  y^{j}  \frac{\partial}{\partial y_1} + \dots + 
  \sum_{1 \leq |j|\le \nu}  a_{j;  n} \, x^{R_{j,n}}  y^{j}  \frac{\partial}{\partial y_n}.  $$
  Let $S = \{y=0\}$.  The holonomy of the foliation generated by $Z_\nu$,
 associated to any small positively oriented simple loop in $S \setminus \{0\}$ and computed 
 on the  transversal $\{x = x_0\}$, where $x_0 \neq 0$,
 is equal to $\mathrm{exp}(2 \pi i Z)|_{x=x_0}$.
 The construction and  the inclusion $Z  (y)  \subset  (y)$ imply that
 \[ j^{\nu} \mathrm{exp}(2 \pi i Z_{\nu})|_{x=1} \equiv j^{\nu} F|_{x=1} \equiv (1,j^{\nu} f),  \]
 that is, the holonomy map of $Z_{\nu}$ on $\{x=1\}$ have the same $\nu$-th jet as $f$.
\qed

\begin{rem} 
\label{rem:y2}
The method provides a slightly more selective version of the system  \eqref{sistema0}. 
Indeed, all components of $G(x,y)$ belong to the ideal $(y)^{2}$.
\end{rem}

\section{Main construction}\label{mainconstruction}\label{secmain}  
In this section we make a sketch of the proof of Theorem \ref{maintro1} and organize it with the aid of some propositions that are proved in next sections.

Theorem \ref{maintro1} follows from the following proposition.

\begin{pro}\label{main}  Let $h\in\diff(\mathbb{C}^n,0)$ and $A \in\gl(n,\mathbb{C})$ be as in Theorem \ref{maintro1}. 
Then  there exists a system of the form \eqref{sistema0} whose holonomy computed in some hyperplane $\{x=x_0\}$,  $x_0\in\mathbb{C}^*$ is given by
$$(x_0,y)\mapsto (x_0, h(y)).$$
\end{pro}

The rest of the section is devoted to outline the proof of this  proposition, for which 
we introduce some simplifications.
By hypothesis, there exists a formal change  of coordinates $\psi$ such that $h$ and $A$ are respectively transformed into  ${\mathfrak h}'$ and ${\mathcal A}$,
where ${\mathcal A}_s = \mathrm{diag} (\mu_1,\hdots, \mu_n)$, ${\mathfrak h}'$ is in normal form with respect to ${\mathcal A}$
and ${\mathfrak h}'$ has no negative resonances. 
Note that 
we can assume that $\left| {\mathcal A}_{N}\right| \leq \epsilon$ for  $\epsilon>0$
arbitrarily chosen
(this can be accomplished by a linear change of coordinates that gets ${\mathcal A}$ in Jordan normal form
with the non-vanishing coefficients outside the diagonal equal to $\epsilon$ instead of the usual $1$, see Remark \ref{rem:jordan}). Fix
$\delta_0, \delta_1 >0$ such that 
$- \delta_1 < \re (\mu_j) < - \delta_0 $
for every $1 \leq j \leq n$. 
Then, by choosing
$\epsilon>0$ sufficiently small, we obtain that
\begin{align}\label{adonai}  (- \delta_1 +\epsilon)  |y|^{2} <  \re \langle {\mathcal A} \cdot y,y\rangle  <
(- \delta_0 -\epsilon)  |y|^{2} \end{align}
for all $y \in {\mathbb C}^n$, where $\langle\,\,,\,\rangle$ denotes the standard Hermitian inner product.

 Fix $\nu \in {\mathbb N}$ such that $$\delta_0 (\nu +1) - \delta_1 \geq 4.$$
By Proposition \ref{pro:positive_construction}
there exists a polynomial vector field $Z$ of the form (\ref{sistema0}),
 with $\mathcal{A}$ instead of $A$,  
whose holonomy map ${\mathfrak h}_0$ on $\{x=1\}$ satisfies $j^{\nu} {\mathfrak h}_0 \equiv  j^{\nu} {\mathfrak h}'$.
By considering a change of coordinates $\phi \in \diff(\mathbb{C}^n,0)$ such that $j^{\nu} \phi \equiv j^{\nu} \psi$, 
we obtain ${\mathfrak h} \in \diff(\mathbb{C}^n,0)$ analytically conjugated to $h$ and such that 
$j^{\nu} {\mathfrak h} \equiv j^{\nu} {\mathfrak h}' \equiv j^{\nu} {\mathfrak h}_0$.
As  the major part of the  proof of  Proposition \ref{main},  we deal with the analogous problem for 
${\mathfrak h}$ and ${\mathcal A}$.

\subsection*{An abstract foliation with the prescribed holonomy}
 We start considering the  singular holomorphic foliation $\mathcal{L}$ be defined by  $Z$ whose holonomy on $\{x=1\}$ is equal to
 ${\mathfrak h}_0$.

The heuristic idea of what we do in the rest of this subsection can be compared to the following surgery to the foliation $\mathcal{L}|_{\mathbb{C}^*\times \mathbb{C}^n}$:  first we cut $\mathbb{C}^*\times \mathbb{C}^n$ along the real hypersurface
$(0,+\infty)\times \mathbb{C}^n$ in order to generate two faces diffeomorphic to 
$(0,+\infty)\times \mathbb{C}^n$. Then we glue them back together taking care to preserve  the foliation but now generating a holonomy 
given by $\mathfrak{h}$,  instead of $\mathfrak{h}_0$.  To begin with our construction, we
consider the diffeomorphism 
\begin{align}\label{metro} \xi\colon = {\mathfrak h}_0 \circ \mathfrak{h}^{-1},
\end{align}
which clearly satisfies $j^{\nu} \xi \equiv \mathrm{id}$.
Consider the  slit plane 
$$\widehat{\mathbb{C}}=\mathbb{C}\backslash (-\infty,0]$$ and the slit disc 
$$\widehat{\mathbb{D}}(2)=\{x\in\widehat{\mathbb{C}}\colon |x|<2 \}.$$
As described next, for some ball $B$ centered at the origin of $\mathbb{C}^n$ we
 can define a unique  holomorphic map $$\Phi\colon {\widehat{\mathbb{D}}(2)}\times B\to\widehat{\mathbb{D}}(2)\times\mathbb{C}^n$$ with the following properties:
 \begin{enumerate}
 \item  $\Phi$  has the form 
$\Phi=(x,\phi(x,y))$;
\item  $\phi(1,y)=\xi(y)$;
\item  $\Phi$ leaves the foliation $\mathcal{L}$ invariant. 
 \end{enumerate} 
 Let us see. Let $\mathcal{L}'$ be the restriction 
of $\mathcal{L}$ to $$\widehat{\mathbb{C}}\times\mathbb{C}^n.$$ Given $x_0\in {\widehat{\mathbb{D}}(2)}$, if $y_0\in  \mathbb{C}^n$ is small, 
the leaf of $\mathcal{L}'$ through $(x_0,y_0)$ 
meets the hyperplane   $\{x=1\}$ at the point 
  $$  \mathrm{exp} (- \log x_0 \, Z) (x_0, y_0), $$
  where $\log $ is the principal determination of the logarithm. 
  If $|y_0|$ were small enough,  the point $\xi ( \mathrm{exp} (- \log x_0 \, Z) (x_0, y_0))$
  would be well defined and
   the leaf of $\mathcal{L}'$ through it would meet 
   the hyperplane $\{x=x_0\}$ at the point 
    \[ \mathrm{exp} (\log x_0 Z)(\xi ( \mathrm{exp} (- \log x_0 \, Z) (x_0, y_0))). \]
    We will see (Lemma \ref{expon}) that this point is actually well defined whenever 
    $|y_0|$ is small enough, independently of $x_0\in {\widehat{\mathbb{D}}(2)}$.
    Thus, if $B$  is a small enough ball centered at the origin of $\mathbb{C}^n$ we can define 
     \begin{equation}
     \label{equ:phi}
      \Phi (x,y) =  \mathrm{exp} (\log x Z)(\xi ( \mathrm{exp} (- \log x \, Z) (x, y))) 
      \end{equation}
     for $x \in \widehat{\mathbb{D}}(2)$ and $y \in B$. Then 
    it is easy to see that  $\Phi$ has the form $\Phi=(x,\phi(x,y))$ and satisfies the properties stated above.  We also can define $\phi_{\scriptscriptstyle -1}$ such that
\begin{align}\label{phi2}{ (x, \phi_{\scriptscriptstyle -1}}(x,y)) = \mathrm{exp} (\log x Z)(\xi^{-1} ( \mathrm{exp} (- \log x \, Z) (x, y)))  \end{align}
for $x \in \widehat{\mathbb{D}}(2), y\in B$. 
 Moreover, it is easy to check the following identities \begin{equation}\label{phi-1}\phi(x,\phi_{-1}(x,y))=y,\quad{\phi_{\scriptscriptstyle -1}}(x,{\phi}(x,y))=y,\quad  x\in \widehat{\mathbb{D}}(2), y\in B.  \end{equation}
 Consider the open sets
\begin{align*}
\Delta_1=\mathbb{D}\backslash [0,-i],\quad\quad {\Delta}_2=\mathbb{D}\backslash [0,i],
\end{align*} which cover  the punctured unitary disc
$\mathbb{D}^*$. We regard the sets $\Delta_1\times B$ and 
${\Delta}_2\times B$ as being disjoint and we construct a complex manifold $\mathcal{M}$ by gluing them  in the following way:
\begin{enumerate}
\item If $x\in\mathbb{D}^*$, $\re (x)<0$ and $y\in B$, we identify $(x,y)\in \Delta_1\times B$ with
 $(x,y)\in\Delta_2\times B$;
\item  If $x\in\mathbb{D}^*$, $\re (x)>0$, $y\in B$ and $\phi(x,y)\in B$,  we identify $(x,y)\in\Delta_1
\times B$ 
with $(x,\phi(x,y))\in {\Delta}_2\times B$. 
\end{enumerate}
 After the gluing process, the sets $\Delta_1\times B$ and $\Delta_2\times B$ become open sets in $\mathcal{M}$, which will be  called $\mathcal{M}_1$ and $\mathcal{M}_2$, respectively. The open set $\mathcal{M}_1$ carries a natural chart
 $\psi_1\colon \mathcal{M}_1\to \Delta_1\times B$ such that $\psi_1(p)=(x,y)$ if and only if $p$ comes from the point $(x,y)\in \Delta_1\times B$ before the gluing process. In the same way, we have  a natural chart $\psi_2\colon \mathcal{M}_2\to \Delta_2\times B$. The coordinate transformation 
 $$\Psi=\psi_2\circ\psi_1^{-1}\colon \psi_1(\mathcal{M}_1\cap \mathcal{M}_2)\to \psi_2(\mathcal{M}_1\cap \mathcal{M}_2)$$ is essentially given by the gluing rules above: the set
  $$\psi_1(\mathcal{M}_1\cap \mathcal{M}_2)\subset \Delta_1\times B$$ is the union of the disjoint open sets 
  $$\mathcal{D}^-=\{(x,y)\in \Delta_1\times B\colon \re (x)<0\}$$ and 
   $$\mathcal{D}^+=\{(x,y)\in \Delta_1\times B\colon \re (x)>0, \ \phi(x,y)\in B\}.$$ Then

 \begin{equation}\label{cambio} \Psi(x,y)=\begin{cases}(x,y) &\mbox{if }  (x,y)\in
 \mathcal{D}^-\\ (x,\phi(x,y)) &\mbox{if } (x,y)\in \mathcal{D}^+. \end{cases}\end{equation} 
 Recall that $\Phi=(x,\phi(x,y))$ preserves the leaves of the foliation $\mathcal{L}$. This means that
   the foliation 
 $\mathcal{L}$ is respected by the gluing process, so  the manifold $\mathcal{M}$ inherits a holomorphic foliation, which is given by the pullback $\psi_1^*(\mathcal{L})$ on $\mathcal{M}_1$, and the 
 the pullback $\psi_2^*(\mathcal{L})$ on $\mathcal{M}_2$. We denote this foliation by 
 $\mathcal{F}^{\scriptscriptstyle\mathcal{M}}$. The set defined in both coordinates $\psi_1$ and $\psi_2$ by the equation $y=0$
is a leaf of  $\mathcal{F}^{\scriptscriptstyle \mathcal{M}}$, which we denote by $L$. 
The following proposition follows immediately from the construction of $\mathcal{F}^{\scriptscriptstyle \mathcal{M}}$. 
\begin{pro} \label{hojal}The  leaf $L$ is closed in $\mathcal{M}$ and holomorphically equivalent to the punctured disc $\mathbb{D}^*$. Moreover,  the holonomy map of $L$
is  holomorphically conjugated to $\mathfrak{h}$.
\end{pro}

\subsection*{Realization as a real foliation of a neighborhood of $0\in\mathbb{C}^{n+1}$. }
 As a second step, our aim is to construct  a particular  $C^\infty$ diffeomorphism between $\mathcal{M}$ and 
a neighborhood of $0\in\mathbb{C}^{n+1}$ with the vertical hyperplane $$E\colon=\{0\}\times \mathbb{C}^n$$ removed. This diffeomorphism induces
 a smooth foliation  that will extend in the class $C^{3}$ 
 to a foliation of a neighborhood of 
 $0\in\mathbb{C}^{n+1}$.  Precisely, we have the following proposition.
 \begin{pro}\label{smoothrealization} There  exist an open neighborhood $\mathcal{M}'$ of $L$ in $\mathcal{M}$,  a neighborhood $U$  of  $0\in\mathbb{C}^{n+1}$ and a smooth diffeomorphism
 $F\colon \mathcal{M}'\to U\backslash E$ such that the following properties hold:
 \begin{enumerate}
  \item \label{fro0.5} For some $r>0$,   the set
   $$U(r)\colon=\{(x,y)\in\mathbb{D}\times\mathbb{C}^n\colon  |y|<r\}$$ is contained in $U$
   \item \label{hojaL}The leaf $L$ is mapped by $F$ onto $\mathbb{D}^*\times\{0\}.$
    \item \label{fro} $F$ is holomorphic outside the set
  $$\psi_1^{-1}\left(\left\{-\frac{\pi}{4}\le\arg (x) \le\frac{\pi}{4},\quad y\in B \right\}\right).$$
 \item \label{fro11} Denote by $\mathcal{F}^{\scriptscriptstyle U}$ the push-forward of $\mathcal{F}^{\scriptscriptstyle \mathcal{M}}$ by $F$.  Then, there exist two vector fields $Y_1$ and $Y_2$
  on $U(r)$ such that:
 \begin{enumerate}
 \item \label{fro0.8} they are of class $C^{2}$;
 \item \label{fro1} they are both tangent to $\mathcal{F}^{\scriptscriptstyle U}$ on
  $U(r)\backslash E$; 
  \item \label{fro2} they are real-linearly independent on $U(r)\backslash\{0\}$;
   \item \label{fro2.1} they are tangent to the vertical hyperplane  $E$;
  \item \label{fro3}they are singular at the origin and we have 
  $$dY_1(0)=
  \begin{bmatrix}1&0\\
                          0& {\mathcal A}
\end{bmatrix}.$$
 \end{enumerate} In particular, the existence of such vector fields means that
 the foliation $\mathcal{F}^{\scriptscriptstyle U}$  extends as a $C^{3}$  foliation on 
 $U(r)\backslash\{0\}$, which leaves $E$ invariant. We denote this extension also by $\mathcal{F}^{\scriptscriptstyle U}$.
 \end{enumerate}
 \end{pro}
\subsection*{Holomorphic realization of $\mathcal{F}^{\scriptscriptstyle U}$. } 
Now, our goal is to find suitable coordinates where the foliation $\mathcal{F}^{\scriptscriptstyle U}$ becomes holomorphic. The almost complex structure of $\mathcal{M}$ induces a $C^\infty$ almost complex structure on
$U\backslash E$, which we denote by $J$. The key step of the construction is the extension of $J$ to
$U(r)$. 
\begin{pro} \label{jexten}The almost complex structure $J$ extends to $U(r)$ as an almost complex structure of class 
$C^{2}$. Moreover, this structure coincides with the canonical complex structure at
each point of $E\cap U(r)$.
\end{pro} 

After this proposition, we are able to apply the following theorem due to Nijenhuis and Woolf about the integrability of almost complex structures.
\begin{thm}[\cite{nw}] \label{teonw}Let $J$ be an almost-complex structure on a manifold $X$ of real dimension $2m$ and of class $C^{k,\alpha}$, where $k$ is a natural number and $\alpha\in(0,1)$. Let the torsion $[J, J]$ of this structure vanish. Then every point of $X$ has a neighborhood on which there exist
 functions $g_1,\dots, g_m$ of class $C^{k+1,\alpha/m}$, which form a complex coordinate system compatible with $J$.
\end{thm}
Thus, with the aid of this theorem we prove the following proposition.
\begin{pro} \label{new}There exist an open set $U'\subset U(r)$ of the form 
$$U'=\{|x|<r', |y|<r'\},\quad r'>0$$
and a diffeomorphism 
 $G\colon U' \to \Omega$, $G(0)=0,$
 where $\Omega$ is a neighborhood of
 $0\in\mathbb{C}^{n+1}$, such that the following properties hold:
\begin{enumerate}
\item \label{new1}$G$ is of class $C^{2, 1/(n+1)}$ and in particular $C^{2}$.
\item \label{new2}$G$ takes the structure $J$ on a neighborhood of $0\in\mathbb{C}^{n+1}$ to the canonical almost complex structure of $\mathbb{C}^{n+1}$.
\item \label{new3}$G$ maps  $\{y=0\}\cap U'$ onto $\{y=0\}\cap\Omega$ and $E\cap U'$ onto $E\cap \Omega$, respectively.
\item \label{new4}$dG(0)=\id$.
\item \label{new5} The foliation $\mathcal{F}^{\scriptscriptstyle U}$ on ${U'\backslash\{0\}}$ is taken by $G$ to a holomorphic foliation  $\mathcal{F}$
on $\Omega\backslash\{0\}$.
\item \label{new6} The foliation $\mathcal{F}$  is generated by a 
  holomorphic vector field
 $Y$ on $\Omega$, with a unique singularity at the origin, such that 
 \[ dY(0) =
  \begin{bmatrix}1&0\\
                          0& {\mathcal A}
\end{bmatrix}. \]
 \item \label{new7}The curve $S=\{y=0\}$ is invariant by $Y$ and its holonomy is holomorphically conjugated to $\mathfrak{h}$. 
\end{enumerate}
\end{pro}

  Now, we are able to make the 
    proof of Proposition \ref{main}.
    \subsection*{Proof of Proposition \ref{main}} Let $\mathcal{F}$ and $Y$ as given by the previous proposition. 
    Up to a linear change of coordinates in the $y$ variables, we can suppose 
    \begin{equation}
    \label{equ:sety}
     j^{1} Y = x \frac{\partial}{\partial x} + A \cdot y \frac{\partial}{\partial y} 
     \end{equation}
     by  \eqref{new6} of Proposition  \ref{new}.
    For $x_0\in\mathbb{C}^*$ small enough, 
     the holonomy of $S=\{y=0\}$ can be computed on
     $\{x=x_0\}$ and 
   is given in the form
    $$h_S\colon (x_0,y)\mapsto (x_0,g(y)),$$
    where $g$ is a local diffeomorphism in $\diff(\mathbb{C}^n,0)$. From \eqref{new7}
    of Proposition \ref{new}, $g$ is holomorphically conjugated to ${\mathfrak h}$ and hence to $h$.
    So, there exists
     a local diffeomorphism $\varphi\in\diff(\mathbb{C}^n,0)$ such that 
    $$\varphi\circ g\circ \varphi^{-1}= h.$$  Since the matrix $d\varphi(0)$ is invertible, we can 
    find $B\in \gl (n,\mathbb{C})$ such that 
    $$e^{x_0 B}=d\varphi(0).$$ Then, if $\varphi$ is well defined and injective on a ball 
    $$B_\epsilon=\{y\in\mathbb{C}^n\colon |y|<\epsilon\}, \quad 0<\epsilon <1,$$ 
    it is easy to verify that   the holomorphic map 
    $$H\colon (x,y)\mapsto \left(x,e^{(x-x_0)B}\cdot \varphi (y)\right)$$ satisfies the following properties:
    \begin{enumerate}
    \item $H$ is injective on $\mathbb{C}\times B_\epsilon$.
    \item $H(x_0,y)=(x_0,\varphi (y))$.
    \item \label{hid} $dH(0)=\id$.
    \end{enumerate}
    In particular, the set 
   $$\Omega\cap \left(\mathbb{C}\times B_\epsilon\right)$$ is mapped by $H$ biholomorphically onto 
   an open set $\tilde{\Omega}$ in $\mathbb{C}^{n+1}$. The pushforward of $\mathcal{F}$ by $H$
   	defines a holomorphic foliation $\tilde{\mathcal{F}}$ on $\tilde{\Omega}$, which is generated by a holomorphic vector field $\tilde{Y}$, which is the pushforward of the vector field $Y$. Recall from 
   	\eqref{fro11} of Proposition \ref{smoothrealization}  that $\mathcal{F}^U$ leaves $\{y=0\}$ and
   	 $\{x=0\}$ invariant. Thus, from \eqref{new3} of Proposition \ref{new} we have that
   	 $\{y=0\}$ and
   	 $\{x=0\}$ are also invariant by $\mathcal{F}$. Then, 
   	since $H$ preserves  $\{y=0\}$ and
   	 $\{x=0\}$, we conclude that $S=\{y=0\}$ and
   	 $\{x=0\}$ are invariant by $\tilde{\mathcal{F}}$. 
   	 In particular,
   	$$\tilde{S}\colon= H(S)$$ is invariant by 
   	$\tilde{\mathcal{F}}$. Since $H$ preserves the set $\{x=x_0\}$, the holonomy of $\tilde{S}$
   	computed in $\{x=x_0\}$ is given by
   	\begin{align*}(x_0,y)\mapsto H\circ h_S\circ H^{-1}(x_0,y)&=H\circ h_S (x_0,\varphi^{-1}(y))\\
   	&=H(x_0,g\circ \varphi^{-1}(y))\\ &=(x_0,\varphi \circ g\circ \varphi^{-1}(y))\\ &=(x_0, h(y)).\end{align*}
   	 From  Equation \eqref{equ:sety} and property \eqref{hid} of $H$ we see that 
   	 $$  j^{1} \tilde{Y} = x \frac{\partial}{\partial x} + A \cdot y \frac{\partial}{\partial y}   .$$ 
	 This together with fact of $\{y=0\}$ and
   	 $\{x=0\}$ being invariant by $\tilde{Y}$ allow us to conclude that $\tilde{\mathcal{F}}$
   	 is defined by a holomorphic system of the form 
   	 \begin{align*}
   	 x'&=x(1+u(x,y))\\
   	 y'&= A\cdot y + v(x,y),
   	 \end{align*} where $u(0,0)=0$, $dv (0)=0$ and  $v(x,0)=0$. 
	 Thus, multiplying by the factor $1/(1+u)$ we conclude that  $\tilde{\mathcal{F}}$ is also defined by a system of the form
   	 \begin{align*}
   	 x'&=x\\
   	 y'&= A\cdot y + G(x,y),
   	 \end{align*} where $G(x,0)=0$ and $dG (0)=0$. 
Therefore Proposition \ref{main} is proved. \qed

We devote the end of the section to Lemma \ref{expon} that was used at the beginning of the section to justify that \eqref{equ:phi} is well-defined.
First, fix $x_0 \in  \widehat{\mathbb{D}}(2)$ and consider $y_0\in {\mathbb C}^n$. 
Recall that $Z$ is a system of the form \eqref{sistema0}, with $\mathcal A$ instead of $A$. 
If the curve
 \begin{align*}w(t)= (x_0^{1-t},y(t)) &= \mathrm{exp} (-t \log x_0 Z) (x_0 , y_0),
   \end{align*}
   which is a solution of the system  
    \begin{align}\label{ezequiel}
       w'  =  -\log x_0 Z(w),
 \end{align}  
 is defined for all $t\in [0,1]$, we define $f_{x_0}(y_0)=y(1)$. The function $f_{x_0}$ represents
 the holonomy map of $\mathcal L'$ from $\{x=x_0\}$ to $\{x=1\} $, it  is a biholomorphism  between neighborhoods of $0\in\mathbb C^n$ that fixes the origin.  
 \begin{lem}
 \label{lem:hol}
  There exist constants $r, c_0,c_1>0$ such that  $f_x$ is defined on the ball $\{|y|< r\}$ 
 for all  $x\in\widehat{\mathbb{D}}(2)$ and
\begin{align}
\label{simon}
c_1|x|^{\delta_1}|u-v|\le|f_x(u)-f_x(v)|\le c_0|x|^{\delta_0}|u-v|
\end{align}
whenever 
 $x\in\widehat{\mathbb{D}}(2)$, $|u|<r$ and $|v|<r$. 
 \end{lem}
 \begin{proof}
  If we express $x=\rho e ^{i\theta}$, 
 $\rho\in (0,2)$, $\theta\in (-\pi,\pi)$, we can regard $f\colon=f_x$ as depending (continuously) on the parameter $(\rho,\theta)\in (0,2)\times (-\pi,\pi)$. The advantage now is that $f$ can be extended
to a family of biholomorphisms depending continuously on the parameter
$(\rho,\theta)\in (0,2]\times [-\pi,\pi]$. Thus, if we fix $\varepsilon\in (0,2)$, from the compactness
of $[\varepsilon,2]\times [-\pi,\pi]$ it is not difficult to see that there exist $r_\varepsilon, c_\varepsilon>0$ such that, for all $x\in  \widehat{\mathbb{D}}(2)$ with
$|x|\ge\varepsilon$, the map
$f_x$ is defined on the ball $\{|y|<r_\varepsilon\}$  and 
$$({1}/{c_\varepsilon}) |u-v|\le |{f_x(u)-f_x(v)}|\le c_\varepsilon |u-v|$$
whenever $ x\in  \widehat{\mathbb{D}}(2)$, $|x|\ge\varepsilon$, $|u|<r_\epsilon$ and  $|v|<r_\epsilon$. 
Therefore, it suffices to  prove the same but with the restriction 
$|x|\le \varepsilon$.

We choose $\varepsilon=e^{-{2\pi |\mathcal A|}/{\delta_0}}$.
Take  $r\in(0,1]$.  Fix $x_0 \in  \widehat{\mathbb{D}}(2)$ with $|x_0|\le\varepsilon$. Consider $u_0,v_0\in {\mathbb C}^n$ with $|u_0|,|v_0|<r$ and suppose that  the curves 
 \begin{align*} (x_{0}^{1-t},u(t)) &= \mathrm{exp} (-t \log x_0 Z) (x_0 , u_0)\quad \textrm{and}\\
   (x_{0}^{1-t},v(t)) &= \mathrm{exp} (-t \log x_0 Z) (x_0 , v_0)
   \end{align*}
   are defined on the interval $[0,s]$ with $s\in(0,1]$. Suppose moreover that
 $|u(t)|, |v(t)|<r$ for all $t\in(0,s]$. Then, since $|x_0|\le 1$,  the curves above are contained in
  $\{|x|\le 1,|y|\le 1\}$.  Recall that 
  --- from Remark \eqref{rem:y2} --- the function $G$ satisfies in the set  $\{|x|\le 1,|y|\le 1\}$
  an inequality of the form $|\partial G/\partial y|\le c |y|$ for some  $c>0$. Then
 $$|G(x_{0}^{1-t},u(t))-G(x_{0}^{1-t},v(t))|\le cr|u(t)-v(t)|,\quad t\in [0,s].$$
 Set $g(t) =|u(t) - v(t)|^2$.  Then $g'=2\re\langle u' - v' ,u-v \rangle$ and hence,
 using the fact that $u$ and $v$ are solutions of \eqref{ezequiel}, we have that
\begin{align*}g'(t)&= 2\re\langle - \log x_0 \mathcal{A} (u-v) , u-v\rangle + 
Q,
\end{align*}
where   
\begin{align*}
\begin{split}
|Q| = & \left|2\re\langle - \log x_0  [G(x_{0}^{1-t},u)-G(x_{0}^{1-t},v)] , u-v\rangle \right|\le \\  
& 2cr (-\log |x_0|+\pi)|u-v|^2. 
\end{split}
\end{align*} 
On the other hand, it follows from \eqref{adonai} that
 \[ (\log |x_0| (\delta_1 - \epsilon) - \pi |\mathcal{A}|) |y|^{2}  \leq \re\langle - \log x_0 \, \mathcal{A}\cdot y,y\rangle 
\leq  (\log |x_0| (\delta_0 + \epsilon) + \pi |\mathcal{A}|) |y|^{2}   \] 
for all $y\in\mathbb C^n$. Therefore, if we fix $r$ such that $r <\min\{\epsilon/c, |\mathcal A|/c\}$,
we obtain 
\begin{align*} 2\eta_1 g(t)\le g'(t)\le 2 \eta_0 g(t), \quad t\in[0,s],
\end{align*}
where  
$\eta_1= \delta_1\log |x_0| -2\pi|\mathcal A|$ and $\eta_0= \delta_0\log |x_0| +2\pi|\mathcal A|<0$.  Thus, we conclude that
\begin{align}\label{levi} e^{\eta_1 t}|u_0-v_0|\le |u(t)-v(t)|\le e^{\eta_0 t} |u_0-v_0|, \quad t\in[0,s],
\end{align}
provided $u$ and $v$ are defined and $|u|,|v|<r$  on $ [0,s] $ ($s\le1$).
Let us show that $u(t)$  is in fact defined on $[0,1]$ and $|u(t)|<r$ for all $t\in[0,1]$ --- and the same for $v$. Otherwise, or the solution $(x(t),u(t))$ is defined on a maximal interval $[0,s)$ with $s\le 1$,
or $u$ is defined on $[0,1]$ and there exists $s_0\in [0,1]$ such that $|u(t)|<r$ for all
$t\in[0,s_0)$ and $|u(s_0)|=r$.  In the first case, since $(x,u)$ is a maximal solution 
of the system \eqref{ezequiel}, it can not be contained in the compact set $\{|x|\le1,|y|\le r\}$, so
again there exists 
$s_0\in [0,1]$ such that $|u(t)|<r$ for all
$t\in[0,s_0)$ and $|u(s_0)|=r$. If  in  \eqref{levi} we set  $v_0=0$, and so  $v\equiv 0$, we obtain
that 
$|u(t)|\le e^{\eta_0 t} |u_0|$ for all $t\in[0,s_0]$, so that $|u(s_0)|\le |u_0|<r$, which is a contradiction.

Thus, we have proved that the holonomy $f_x$ is defined on $\{|y|<r\}$ and, by substituting 
the values of $\eta_0$ and $\eta_1$ in \eqref{levi}, we obtain
$$e^{ -2 \pi|\mathcal A| } |x_0|^{\delta_1}|u_0-v_0|\le |u(1)-v(1)|\le e^{2 \pi|\mathcal A| }|x_0|^{\delta_0} |u_0-v_0|,$$
which finishes the proof.
 \end{proof}
 Finally, we can prove the announced Lemma \ref{expon}.
\begin{lem}\label{expon} There exist $\delta>0$ and $C>0$ such that the maps $\Phi$ and $\Phi^{-1}$ are well-defined and satisfy 
\[  \max (|\phi(x,y) - y|,  |\phi_{-1}(x,y) - y| ) \leq C |x|^{\delta_0 (\nu+1) - \delta_1} |y|^{\nu+1}    \]
whenever $x \in  \widehat{\mathbb{D}}(2)$ and $|y| < \delta$.
\end{lem} 
\proof 
Consider the notations in Lemma \ref{lem:hol}.
By reducing $r$ we can assume that  $\{|y|<c_0 2^{\delta_0} r\}$ is contained in the domain of $\xi$.
Take $\delta=\sigma r$, where $\sigma\in (0,1)$ will be defined later.
 Suppose that $x_0\in\widehat{\mathbb D}(2)$
and $|y_0|<\delta$. Since $\delta<r$, from Lemma \ref{lem:hol} we have that
$f_{x_0}(y_0)$ is defined  and $|f_{x_0}(y_0)|\le c_0 |x_0|^{\delta_0}|y_0|<c_0 2^{\delta_0} r$, whence $\xi\circ f_{x_0}(y_0)$ is defined. In order to prove that $\phi(x_0,y_0)$ is defined
it suffices to show that $\xi\circ f_{x_0}(y_0)$ is in the image of $f_{x_0}$,  in this case we have 
$\phi(x_0,y_0)=f_{x_0}^{-1}\circ\xi\circ f_{x_0}(y_0)$. Let $\mathcal B\subset\mathbb C^n$ be the open ball of radius $(1-\sigma)r$ centered at $y_0$. Since $\overline{\mathcal B}\subset\{|y|<r\}$,
from \eqref{simon} we have 
$$|f_{x_0}(\zeta)- f_{x_0}(y_0)|\ge r_0\colon =c_1|x_0|^{\delta_1}(1-\sigma)r,\quad \zeta\in\partial 
\mathcal B.$$
Then $f_{x_0}(\partial\mathcal B )$ is disjoint of the ball $B_0\colon =\{|y- f_{x_0}(y_0) |<r_0\}$ and so
the linking number of $f_{x_0}(\partial\mathcal B )$ with any point  in $B_0$ is constant. This number is equal to one, because the diffeomorphism $f_{x_0}|_ \mathcal B$ takes the value $ f_{x_0}(y_0)  \in B_0$. Then
$B_0\subset f_{x_0}(\mathcal B )$, so it is enough to prove that 
$\xi\circ f_{x_0}(y_0)\in B_0$ or, equivalently,   that $$\frac{|\xi\circ f_{x_0}(y_0)-f_{x_0}(y_0)|}{r_0}<1.$$ Since $j^{\nu} \xi \equiv \id$, we can assume that there is a
 constant $c_\xi>0$ such that
$|\xi(y)-y|\le c_\xi |y|^{\nu+1}$. Then 
\begin{align*}
\frac{|\xi\circ f_{x_0}(y_0)-f_{x_0}(y_0)|}{r_0} \le \frac{c_\xi|f_{x_0}(y_0)|^{\nu+1}}
{r_0}
\le
\frac{c_\xi  \left( c_0|x_0|^{\delta_0}|y_0|\right)^{\nu+1}}{c_1|x_0|^{\delta_1}(1-\sigma)r}\\
\le \frac{c_\xi c_0^{\nu+1}2^{(\nu+1)\delta_0-\delta_1}r^{\nu}\sigma^{\nu+1}}{c_1(1-\sigma)},
\end{align*}
which is lesser than $1$ if $\sigma$ is chosen small enough.
 Finally, let 
$x\in\widehat{\mathbb D}(2)$
and $|y|<\delta$. We have proved that $\phi(x,y)=f_x^{-1}\circ \xi \circ f_x (y)$ is well defined and  from the proof we have $\phi(x,y)\in\mathcal B$, whence  $|\phi(x,y)|<r$. Then, if we set  $u=\phi(x,y)$ and $v=y$ in
\eqref{simon} we obtain that
\begin{align*}
 c_1|x|^{\delta_1}|\phi(x,y)-y|\le |\xi \circ f_x (y)-f_x(y)|\le c_\xi |f_x(y)|^{\nu+1}
 \le c_\xi \left( c_0|x|^{\delta_0}|y|\right)^{\nu+1},
 \end{align*}
 that is,
 \begin{align*}
 |\phi(x,y)-y|
 \le c_1^{-1}c_\xi c_0^{\nu+1}|x|^{\delta_0(\nu+1)-\delta_0}|y|^{\nu+1}.
 \end{align*}
 The inequality for $|\phi_{-1}(x,y)-y|$ is obtained in a similar way.  \qed

 \section{A special class of functions}\label{secspecial}
  Let $-\pi \leq a < b \le \pi$,  $r_1,r_2\in (0,+\infty]$ and $n\in\mathbb{Z}_{\ge 0}$, and consider the set 
  $$V(a,b,r_1,r_2,n)\colon=\left\{(x,y)\in\mathbb{C}^*\times\mathbb{C}^n\colon a<\arg (x)<b, |x|<r_1, |y|<r_2\right\}$$
 if $a \neq -\pi$ or $b \neq \pi$ and
 $$V(-\pi,\pi,r_1,r_2,n)\colon=\left\{(x,y)\in\mathbb{C}^*\times\mathbb{C}^n\colon |x|<r_1, |y|<r_2\right\}.$$
\begin{defn} \label{definition} Let
$\alpha\in\mathbb{R}$, $p\in\mathbb{Z}_{\ge 0}$  and consider a $C^\infty$ map
$$f\colon V(a,b,r_1,r_2,n)\to \mathbb{R}^m,\quad m\in\mathbb{N}.$$  We say that $f$ is of order $(\alpha, p)$ (at the origin) 
and we denote $f \in O_{\alpha, p}$ if, 
given $k\in\mathbb{Z}_{0 \le k \le p}$ and given $g$ a partial
 derivative of order $k$ of $f$ , the function
 $$|x|^{k-\alpha}g(x,y)$$ is bounded near the origin of $\mathbb{C}^{n+1}$. 
 We say that $f$ is of order $\alpha$, and we denote $f \in O_{\alpha}$,  if it is of order $(\alpha, p)$ for any $p \in {\mathbb N}$.
 \end{defn}
   \begin{pro} \label{clase1}If $f\colon V(a,b,r_1,r_2,n)\to \mathbb{R}^m$ is
  of order $(\alpha, p)$, then 
  \begin{enumerate}
  \item \label{clase1a} $f$ is of order $(\alpha' , p')$ whenever $\alpha'\le\alpha$ and $p' \le p$; and
  \item \label{clase1b}the partial derivatives  of $f$ are  of order $(\alpha-1 , p-1)$ if $p \geq 1$.
  \end{enumerate}
  \end{pro}
  The proposition is a direct consequence of the definition above, so we leave the proof to the reader.

   \begin{pro} \label{clase11}Let $f_1,f_2\colon V(a,b,r_1,r_2,n)\to \mathbb{R}^{m}$ be of 
   order $(\alpha_1, p_1)$ and $(\alpha_2, p_2)$, respectively, and set $\alpha = \min (\alpha_1, \alpha_2)$ and $p = \min (p_1, p_2)$.
   \begin{enumerate}
   \item\label{itemsum} Then $f_1+f_2$ is of order $(\alpha , p)$;
   \item\label{itemprod} If $m=1$, then   $f_1f_2$ is  of order $(\alpha_1+\alpha_2 , p)$;
    \item\label{itemquot} If $m=1$ and $\alpha_2 > 0 $ then   $f_1/(1 + f_2)$ is  of order $(\alpha_1 , p)$.
  \end{enumerate}

  \end{pro}
 
 \begin{proof}

The assertion \eqref{itemsum} is easily verified, so we only deal with \eqref{itemprod} and  \eqref{itemquot}. We proceed by induction on $p$. It is easy to see that \eqref{itemprod} holds true if $p=0$. Suppose that \eqref{itemprod} holds for $p=k\in\mathbb{Z}_{\ge 0}$. Let $f_1,f_2\colon V\to\mathbb{R}$ be respectively of 
order $(\alpha_1,p_1)$ and  $(\alpha_2,p_2)$, such that $p=k+1$. 
Since $f_1f_2$ is obviously of order $(\alpha_1+\alpha_2,0)$, in order to prove that $f_1f_2$ is of order $(\alpha_1+\alpha_2,k+1)$ it suffices to show that each partial derivative of  $f_1f_2$ is of order $(\alpha_1+\alpha_2-1,k)$.
Let $g$ be a partial derivative of $f_1f_2$. Then $g=f_1'f_2+f_1f_2'$, where $f_i'$ is a partial derivative of $f_i$ for $i=1,2$. By proposition \ref{clase1} we have
$f_1'\in O_{\alpha_1-1,p_1-1}$ and $f_2\in O_{\alpha_2,p_2-1}$ and so, by the inductive hypothesis, $f_1'f_2\in O_{\alpha_1+\alpha_2-1,k}$. In the same way 
 $f_1f_2'\in O_{\alpha_1+\alpha_2-1,k}$  and it follows from assertion \eqref{itemsum} that
$g\in O_{\alpha_1+\alpha_2-1,k}$. Now, in view of assertion \eqref{itemprod},
in order to prove \eqref{itemquot} it suffices to show that $1/(1+f_2)$ is of order
$(0,p_2)$ if $\alpha_2>0$. This is proved by induction. It is obviously true that $1/(1+g)\in O_{0,p}$  if
 $g\in O_{\beta,p}$, $\beta>0$ and $p=0$.  Suppose that $1/(1+g)\in O_{0,p}$  if
 $g\in O_{\beta,p}$, $\beta>0$ and $p=k\in\mathbb{Z}_{\ge 0}$. 
 Consider any  $g\in O_{\beta,k+1}$ with $\beta>0$. Then $1/(1+g)\in O_{0,0}$ and so, in order to get  $1/(1+g)\in O_{0,k+1}$,  it suffices to show that each partial derivative of $1/(1+g)$ is of order
 $(-1,k)$. A partial derivative  of $1/(1+g)$ has the form $h=-g'/(1+g)^2$, where
 $g'$ is a partial derivative of $g$. Since $g\in  O_{\beta,k}$, by the inductive hypothesis we have $1/(1+g)\in  O_{0,k}$, whence $1/(1+g)^2\in  O_{0,k}$. Then,
 since $g'\in O_{\beta-1,k}$, it follows from \eqref{itemprod} that
  $h\in O_{\beta-1,k}\subset  O_{-1,k}$. 
\end{proof}
\begin{pro}
\label{pro:ck_extension}
Consider a $C^\infty$ map
$$f\colon V(-\pi,\pi,r_1,r_2,n)\to \mathbb{R}^m,\quad m\in\mathbb{N}$$ 
of class $(\alpha, p)$ with $\alpha>0$. Let $k \in {\mathbb Z}_{\geq 0}$ with $k < \alpha$ and $k \leq p$. 
Then $f$ extends to a $C^{k}$ function in some set $\{ |x| < r_{1}' \} \times \{ |y| < r_{2}' \}$ whose derivatives up to order $k$
vanish on $\{x=0\}$.
\end{pro} 
Notice that any derivative $g$ of $f$ up to order $k$ extend  to $\{x=0\}$ continuously  by defining $g(0, y) \equiv 0$ since $f \in O_{\alpha}$ and $\alpha>k$. 
Thus, the proposition is an immediate consequence of the following lemma. 
\begin{lem} \label{lemac1}Let $U$ be an open subset of $\mathbb{R}^m$ and let $M\subset U$ be 
a proper embedded $C^1$ manifold of codimension $\ge 1$. 
Let $f\colon U\backslash M\to\mathbb{R}^n$ be of  class $C^1$. Suppose that $f$ and its partial derivatives extend continuously to $U$. Then $f$ extends to $U$ in the class $C^1$. 
\end{lem}
\proof 
Let $\bar{f}\colon U\to\mathbb{R}^n$ be the extension of $f$ to $U$. In view of the hypotheses, it
 is enough to prove that, given $p\in M$, we can find a $C^1$ coordinate system 
 $(x_1,\dots, x_m)$ around $p$ such that the partial derivatives of $\bar{f}$ at $p$ exist and are continuous in a neighborhood of $p$.
 Consider affine coordinates such that the canonical unitary vectors
 $e_1,\dots,e_m$ are transverse to $M$ at $p$, and fix $j\in\{1,\dots,m\}$ and $q \in M$ in a neighborhood of $p$.  It suffices to show
\[  \frac{\partial\bar{f}}{\partial x_j}(q)=\lim\limits_{x\rightarrow q}\frac{\partial {f}}{\partial x_j}(x),\quad j=1,\dots,m. \]    
Since $e_j$ is transverse to $M$ at $q$, for $t\in\mathbb{R}^*$ small enough the euclidean segment $[q,q+te_j]$ intersects $M$ only at $q$. 
Thus, if we set $f=(f_1,\dots,f_m)$ and  $\bar{f}=(\bar{f}_1,\dots,\bar{f}_m)$ , 
 by the Mean Value Theorem we have 
 \begin{align*}\frac {\bar{f}(q+te_j)-\bar{f}(q)}{t}&=
 \left(\frac {\bar{f}_1(q+te_j)-\bar{f}_1(q)}{t},\dots, \frac {\bar{f}_m(q+te_j)-\bar{f}_m(q)}{t}\right)\\
 &= \left( \frac{\partial {f}_1}{\partial x_j}(w_1) ,\dots, \frac{\partial {f}_m}{\partial x_j}(w_m)\right),
 \end{align*} where $w_1,\dots, w_m$ are points in the open segment  $$(q,q+te_j)\subset U\backslash M.$$ Then, since the points $w_1,\dots, w_m$ tend to $q$ as $t$ tends to $0$, we have 
 \begin{align*}\lim\limits_{t\rightarrow 0}\frac {\bar{f}(q+te_j)-\bar{f}(q)}{t}&=
 \left(\lim\limits_{x\rightarrow q}\frac{\partial {f}_1}{\partial x_j}(x),\dots, \lim\limits_{x\rightarrow q}\frac{\partial {f}_m}{\partial x_j}(x)\right)\\
 &=\lim\limits_{x\rightarrow q}\frac{\partial {f}}{\partial x_j}(x),
 \end{align*}
 completing the proof.
 \qed
  
\begin{pro} \label{margot}Let 
$$f\colon V(a,b,r_1,r_2,n)\to \mathbb{C}^m$$ be holomorphic such that, for some $\alpha\in\mathbb{R}$,
 the function 
 $|x|^{-\alpha}|f(x,y)|$ is bounded near the origin. Then, if 
 \begin{itemize}
 \item  $[a',b']\subset (a,b)$  or
 \item $a'=a=-\pi$ and $b'=b=\pi$,
 \end{itemize}
 the restriction of $f$
 to $V(a',b',r_1,r_2,n)$ is of order $\alpha$.  
\end{pro}
\proof
Let us consider the first case, the second one is simpler.
Without loss of generality we can assume that there exists $c>0$ such that 
\begin{align}\label{mouse}|x|^{-\alpha}|f(x,y)|\le c\; \textrm{ on } \; V\colon=V(a,b,r_1,r_2,n).
\end{align}
Let $r>0$ such that 
\begin{align}\label{rrra} 2r<\min\{r_1,r_2\},
\end{align}
consider the set
$$U\colon =V(a',b',r,r,n)\subset V$$
and let $$\delta\colon U\to \partial V$$ be the distance to the boundary of $V$. 
 Given ${\mathfrak{z}}=({\mathfrak{x}},{\mathfrak{y}})\in U$, since
the point
$(0,{\mathfrak{y}})$ belongs to $\partial V$, we see that 
\begin{align} \label{di0} \delta({\mathfrak{z}})\le |{\mathfrak{x}}|<r.\end{align} moreover, we can find
${\mathfrak{z}}_{\scriptscriptstyle\partial}=({\mathfrak{x}}_{\scriptscriptstyle\partial},{\mathfrak{y}}_{\scriptscriptstyle\partial})\in \partial V$
 such that $$\delta({\mathfrak{z}})=|{\mathfrak{z}}_{\scriptscriptstyle\partial}-{\mathfrak{z}}|. $$ If $|{\mathfrak{x}}_{\scriptscriptstyle\partial}|=r_1$, using \eqref{rrra}
we have $$\delta({\mathfrak{z}})=|{\mathfrak{z}}_{\scriptscriptstyle\partial}-{\mathfrak{z}}|\ge |{\mathfrak{x}}_{\scriptscriptstyle\partial}-{\mathfrak{x}}|
\ge |{\mathfrak{x}}_{\scriptscriptstyle\partial}|-|{\mathfrak{x}}|\ge r_1-r>r,$$ which contradicts  \eqref{di0}. In the same way we see that $|{\mathfrak{y}}_{\scriptscriptstyle\partial}|\neq r_2$. Then it remain the following possibilities:
\begin{enumerate}
\item $\arg({\mathfrak{x}}_{\scriptscriptstyle\partial})=a$;
\item $\arg({\mathfrak{x}}_{\scriptscriptstyle\partial})=b$;
\item ${\mathfrak{x}}_{\scriptscriptstyle\partial}=0$.
\end{enumerate}
 Thus, since $\arg({\mathfrak{x}})\in (a',b')$, in the first case above we have 
 $$\delta({\mathfrak{z}})=|{\mathfrak{z}}_{\scriptscriptstyle\partial}-{\mathfrak{z}}|\ge |{\mathfrak{x}}_{\scriptscriptstyle\partial}-{\mathfrak{x}}|\ge|{\mathfrak{x}}|\Big|\sin \big(\arg({\mathfrak{x}}_{\scriptscriptstyle\partial})-\arg({\mathfrak{x}})\big)\Big|\ge |{\mathfrak{x}}||\sin (a-a')|.$$
 Analogously, in the second case we have 
 $$\delta({\mathfrak{z}})\ge |{\mathfrak{x}}||\sin (b-b')|.$$ Finally, in the third case we easily see that
 $$\delta({\mathfrak{z}})= |{\mathfrak{x}}|.$$ Therefore, if $$\epsilon\colon = 
 \min\Big\{|\sin (a-a')|, |\sin (b-b')|, 1\Big\},$$ we conclude that
 \begin{align}\label{minz}\delta({\mathfrak{z}})\ge \epsilon|{\mathfrak{x}}|,\quad {\mathfrak{z}}\in U.
 \end{align} 
 On the other hand, given ${\mathfrak{z}}\in U$, let $B_{\mathfrak{z}}\subset\mathbb{C}^{n+1}$ be the ball of radius $\delta({\mathfrak{z}})/2$ centered at ${\mathfrak{z}}$.
 By the definition of $\delta$ the ball $B_{\mathfrak{z}}$ is contained in $V$, so $f$ is defined on $B_{\mathfrak{z}}$ and from
 \eqref{mouse},
 \begin{align}\label{charge}|f(x,y)|\le c|x|^\alpha, \quad (x,y)\in B_\mathfrak{z}. \end{align}
If $(x,y)\in B_\mathfrak{z}$, using \eqref{di0} we see that
$$|x|\le |\mathfrak{x}|+\delta(\mathfrak{z})\le 2|\mathfrak{x}|,$$ so from \eqref{charge}
we obtain 
 \begin{align}\label{charge1}|f(x,y)|\le  2^\alpha c |\mathfrak{x}|^\alpha, \quad (x,y)\in B_\mathfrak{z}. \end{align} 
 Therefore, it follows from the Cauchy Derivative Estimates that, given $k\in\mathbb{N}$, 
 there exists a constant $c(k)>0$ depending only on $k$ such that for any 
  partial derivative $g$ of order $k$ of $f$,  
 \begin{align}\label{cde}|g(\mathfrak{z})|\le c(k) \frac{2^\alpha c|\mathfrak{x}|^\alpha}{(\delta(\mathfrak{z})/2)^k}.\end{align}
 Thus, from this and \eqref{minz} we obtain
  \begin{align}\label{cde1}|\mathfrak{x}|^{k-\alpha}|g(\mathfrak{z})|\le c(k) \frac{2^\alpha c}{(\epsilon/2)^k},\quad \mathfrak{z}\in U,\end{align} which shows that $f|_U$ has order $\alpha$. \qed

 \begin{lem}\label{argumento}Let $\rho\colon [a,b]\to\mathbb{R}$ be of class $C^\infty$, where 
   $[a,b]\subset [-\pi,\pi]$. Consider the function $f(z)=\rho (\arg(z))$ defined on the set 
   $$V=\{z\in\mathbb{C}\colon a<\arg(z)<b\}.$$ Then $f$ is of order $0$. 
 \end{lem}
  \proof
 We prove first that $$\arg\colon V\to \mathbb{R}$$ is of order $0$. Since this function is bounded,
 it is enough to prove that its partial derivatives are of order $-1$. Thus, since  
 $$\frac{\partial\arg}{\partial x}(z)=\im({1}/{z}),\quad
 \frac{\partial\arg}{\partial y}(z)=\re({1}/{z}),\quad z=x+iy,$$
 it suffices to show that the function $1/z$ on $V$ is of order $-1$, which is a direct consequence of
 Proposition \ref{margot}.
 Let us prove that $f=\rho \circ \arg$ is of order $0$. Since $f$ is bounded, we only have to deal with
 the partial derivatives of order $k\in\mathbb{N}$ of $f$. Fix $k\in\mathbb{N}$ and let $g$ be a partial derivative of order $k$ of $f$. It is not difficult to see that $g$ can be expressed as finite summation 
  \begin{align}\label{sierva}g(z)=\sum\limits_{j}\rho^{(k_j)}{\big(\arg(z)\big)}u_{j1}u_{j2}\cdots u_{jk_j},
 \end{align} where
 \begin{enumerate}
 \item $k_j\in\{1,\dots,k\}$;
 \item $u_{jl}$ is a partial derivative of order $p_{jl}$ of $\arg(z)$, where $l=1,\dots,k_j$;
 \item $p_{j1}+p_{j2}+\cdots p_{jk_j}=k$.
 \end{enumerate}
  Since the function $\arg(z)$ is of order $0$, the function $|z|^{p_{jl}}u_{jl}$ is bounded near the origin for
  each $l=1,\dots, k_j$.
  Then $$\left(|z|^{p_{j1}}u_{j1}\right)\cdots \left(|z|^{p_{jk_j}}u_{jk_j}\right)$$ is
  bounded near the origin, that is, 
  $$|z|^{k}u_{j1}\cdots u_{jk_j}$$ is bounded near the origin, whence 
  $$|z|^{k}\rho^{(k_j)}{\big(\arg(z)\big)}u_{j1}\cdots u_{jk_j}$$ is bounded near the origin, because
  $\rho^{(k_j)}$ is bounded. Therefore, it follows from \eqref{sierva} that $|z|^{k}g(z)$ is bounded near the origin. We
  conclude that $f$ is of order $0$. \qed

Now, we define a special class of diffeomorphisms. 
\begin{defn}
Let $F:V(a,b,r_1,r_2,n)   \to W$ be a diffeomorphism between open sets of ${\mathbb C}^{n+1}$, and consider $\beta>1$ and $p \in {\mathbb Z}_{\geq 0}$.
We say that $F$ is a diffeomorphism of order $(\beta,p)$, and we denote $F \in D_{\beta, p}$,  
if $x \circ F \equiv x$ and $F - \mathrm{id}$ is of order
 $(\beta, p)$ in the sense of Definition \ref{definition}. We say that $F$ is of order $\beta$ and
we denote $F \in D_{\beta}$ if $F \in D_{\beta, p}$ for any $p \in {\mathbb N}$.
\end{defn}
We want to consider the composition of diffeomorphisms in $D_{\beta}$ with maps in $O_{\alpha}$.
\begin{pro} \label{pro:comp_O_D}
Consider 
\[ f: V(a,b,r_1,r_2,n)  \to {\mathbb R}^{m} \quad \mathrm{and} \quad  F:  V(a,b,r_{1}',r_{2}',n) \to W \]
that belong to $O_{\alpha,p}$ and $D_{\beta, p}$ respectively with $p \geq 1$.  Then, there exist $\tilde{r}_1, \tilde{r}_2 >0$.
such that $f \circ F$ is defined in $ V(a,b,\tilde{r}_1,\tilde{r}_2,n)$  and is of order $(\alpha, p)$. 
\end{pro}
\begin{proof}
Since $F - \mathrm{id} \in O_{\beta, 0}$, it is easy to find $\tilde{r}_1, \tilde{r}_2 >0$ such that 
\[ F( V(a,b,\tilde{r}_1,\tilde{r}_2,n) ) \subset  V(a,b,r_1,r_2,n) . \]
We can assume $m=1$ without lack of generality. Any partial derivative of order $k \leq p$ is a finite sum of expressions of the form 
\[ (u \circ F)  v_1 \hdots \ v_s \]
where $u$ is a partial derivative of $f$ of order $r$, each $v_j$ is a partial derivative of a component of $F$ or order $r_j \geq 1$ and 
\[ (r_1 -1) + \hdots + (r_s -1) + r = k . \]
The first derivatives of components of $F$ are of the form $\upsilon +  h$ where $\upsilon \in \{0,1\}$ and $h \in O_{\beta-1, p-1}$ and thus
they belong to $O_{0, p-1}$. In particular $v_j \in O_{1-r_j, p-r_j}$ holds for any $1 \leq j \leq s$.
Moreover, since $u \in O_{\alpha -r, p-r}$, we deduce that
 $u \circ F \in O_{\alpha -r, 0}$. Then $(u \circ F)  v_1 \hdots \ v_s$
is of order 
\[ \left( (\alpha -r) + \sum_{j=1}^{s} (1-r_j), 0 \right)  = (\alpha - k, 0).  \] 
Since such property holds for any $0 \leq k \leq p$,  we conclude that $f \circ F$ is of order $(\alpha, p)$.
\end{proof}
\begin{pro} 
\label{pro:discrete_derivative}
Let $f: W \to {\mathbb R}^{m}$ be a $C^{\infty}$ map defined in an open neighborhood $W$ of the origin in ${\mathbb C}^{n+1}$ 
and consider $F \in D_{\beta, p}$ with  $\beta > 1$ and $p \geq 1$, defined in some  $V(a,b,r_1,r_2,n)$.   Then $f \circ F -  f \in O_{\beta, p}$.
\end{pro}
\begin{proof} 
We can assume $m=1$ without lack of generality. Let ${\mathcal E}$ be the set of $C^{\infty}$ functions 
defined in a neighborhood of the origin in ${\mathbb C}^{n+1}$. It is obvious that $g \in O_0$ for any $g \in {\mathcal E}$. 
So $g \circ F \in O_{0,p}$ for any $g \in {\mathcal E}$ by Proposition \ref{pro:comp_O_D}. Moreover, we have
\[ g \circ F - g \in O_{\beta, 0} \]
for any $g \in {\mathcal E}$ by an application of the mean value theorem.
The proof of the Proposition is completed by an induction process. Let
$q \in {\mathbb Z}_{0 \le q < p}$ and suppose that
$g \circ F - g \in O_{\beta, q}$ for any $g \in {\mathcal E}$. Let us show that 
$g \circ F - g \in O_{\beta, q+1}$ for any $g \in {\mathcal E}$.
Let $(u_1, \hdots, u_{2n+2})$ be real coordinates in ${\mathbb C}^{n+1}$ where $u_1 =\re (x)$, $u_2 = \im (x)$, $u_3 = \re (y_1)$ and so on. 
Let $g \in {\mathcal E}$ and consider 
\[ \Delta_j := \frac{\partial}{\partial u_j} \left(g \circ F - g \right)  = \sum_{l=1}^{2n+2} \left( \frac{\partial g}{\partial u_l} \circ F\right) \frac{\partial (u_l \circ F)}{\partial u_j} - \frac{\partial g}{\partial u_j} . \]
It suffices to show that $\Delta_j \in O_{\beta -1, q}$ for every $1 \leq j \leq 2n+2$. We have
\[ \left( \frac{\partial g}{\partial u_l} \circ F\right) \frac{\partial (u_l \circ F)}{\partial u_j}  \in O_{\beta -1, p-1} \subset O_{\beta -1, q} \]
if $j \neq l$ since the first term of the product belongs to $O_{0,p}$ whereas the second one belongs to $O_{\beta -1, p-1}$.
So it suffices to prove that 
\[  \left( \frac{\partial g}{\partial u_j} \circ F\right) \frac{\partial (u_j \circ F)}{\partial u_j} - \frac{\partial g}{\partial u_j}  =  
  \frac{\partial g}{\partial u_j} \circ F  - \frac{\partial g}{\partial u_j}  + \left( \frac{\partial g}{\partial u_j} \circ F\right) \frac{\partial (u_j \circ F - u_j)}{\partial u_j}   \] 
belongs to $O_{\beta -1, q}$ for $1 \leq j \leq 2n+2$. On the one hand, the term 
\[ \left( \frac{\partial g}{\partial u_j} \circ F\right) \frac{\partial (u_j \circ F - u_j)}{\partial u_j}  \]
belongs to  $O_{\beta -1, p-1}$ and then to $O_{\beta -1, q}$ analogously as above.  On the other hand, we have
\[  \frac{\partial g}{\partial u_j} \circ F  - \frac{\partial g}{\partial u_j}  \in O_{\beta, q}  \subset O_{\beta -1, q}  \]
by the inductive hypothesis, completing the proof.
\end{proof}
Our next goal is to prove that $F \in D_{\beta, p}$, $\beta >1$ and $p \ge 1$ imply $F^{-1} \in D_{\beta, p}$.
The first step is the following topological lemma. 
\begin{lem}\label{perturbacion} Given $r>0$,  denote by $B(r)$ the ball of radius $r$ centered at the origin of $\mathbb{C}^n$. Let   $f\colon B(\mathfrak{r})\to \mathbb{C}^n$, $\mathfrak{r}>0$  be a function
of the form $f(y)=y+\sigma(y)$, where  $\sigma \colon B(\mathfrak{r})\to \mathbb{C}^n$ is differentiable and  $|d\sigma|\le \lambda$  for some $\lambda\in [0,1)$. 
Then $f(B(\mathfrak{r}))$ is open and $f\colon B(\mathfrak{r})\to f(B(\mathfrak{r}))$ is a diffeomorphism.
Furthermore, if $f(0)=0$,  the image of $f$ contains the ball $B\big((1-\lambda)\mathfrak{r}\big)$.
\end{lem}
\proof
Since $|d\sigma|\le \lambda$, from the Mean Value Inequality we see that 
\begin{align}
|\sigma(y_1)-\sigma(y_2)|\le \lambda|y_1-y_2|, \quad y_1,y_2\in B(\mathfrak{r}).
\end{align}
Then 
\begin{align}\nonumber
|f(y_1)-f(y_2)|&=|(y_1-y_2)+(\sigma(y_1)-\sigma(y_2))|
\ge |y_1-y_2|-|\sigma(y_1)-\sigma(y_2)|,
\end{align}
\begin{align}\nonumber
|f(y_1)-f(y_2)|
&\ge  (1- \lambda)|y_1-y_2|, \quad y_1,y_2\in B(\mathfrak{r}).
\end{align}
Then, since $1-\lambda>0$, the inverse $f^{-1}$ exists and it is continuous. Thus  $f$ is a homeomorphism
and therefore  $f(B(\mathfrak{r}))$ is open. Finally, since $|d\sigma|\le \lambda$ guarantees that  $df=\id+d\sigma$ is invertible, we conclude that $f$ is a diffeomorphism, which proves the first assertion of the lemma. Suppose now that $f(0)=0$, that is, $\sigma(0)=0$.  Assume by contradiction that there exists 
$p\in B\big((1-\lambda)\mathfrak{r}\big)$ such that $p\notin f(B(\mathfrak{r}))$. 
Take any $\mathfrak{r}'<\mathfrak{r}$ such that \begin{align} \label{prr}|p|<(1-\lambda)\mathfrak{r}'.
\end{align}
Firstly, note that since  
 $p\notin f(B(\mathfrak{r}'))$, the linking number of $f(\partial B(\mathfrak{r}'))$ and $p$ is zero. 
On the other hand, the function 
$$H(y,t)=y+t\sigma(y),\quad y\in \partial B(\mathfrak{r}'), t\in [0,1]$$
is a homotopy 
between $\partial B(\mathfrak{r}')$ and $f\big(\partial B(\mathfrak{r}')\big)$. Since 
\begin{align*}
\left| H(y,t)\right| &=|y+t\sigma(y)|\ge|y|-|\sigma(y)|= |y|-|\sigma(y)-\sigma(0)|\\
&\ge  |y|-\lambda |y|= (1-\lambda)\mathfrak{r}'\\
&>|p|,
\end{align*}
$H$ is indeed a homotopy in $\mathbb{C}^n\backslash\{p\}$. This implies that the linking number
of $\partial B(\mathfrak{r}')$ and $p$ coincides with the linking number of $f\big(\partial B(\mathfrak{r}')\big)$ and $p$, which --- as we have seen above --- is equal to zero. That is, the linking number of 
$\partial B(\mathfrak{r}')$ and $p$  is zero. But this is a contradiction because from \eqref{prr} we 
see that $p\in B(\mathfrak{r}')$, whence the linking number of $\partial B(\mathfrak{r}')$ and $p$
is actually equal to one. \qed
\begin{rem}
Lemma \ref{perturbacion} implies that if $F:  V(a,b,r_{1}, r_{2},n) \to W$ belongs to $D_{\beta, p}$ ($\beta >1$, $p \geq 1$) 
then there exist $r_{1}' , r_{2}' >0$
such that $V(a,b,r_{1}' , r_{2}',n)$ is contained in the image of $F$ and so $F^{-1}$ is defined in such set.
\end{rem} 
\begin{pro}
\label{pro:inverse}
Let $F \in D_{\beta, p}$ with $\beta > 1$ and $p \geq 1$. Then $F^{-1} \in D_{\beta, p}$.
\end{pro}
\begin{proof} 
Since $F$ is of class $C^{1}$ so is $F^{-1}$. The property $F - \mathrm{id} \in O_{\beta,p}$ implies 
$F^{-1} - \mathrm{id} \in O_{\beta,0}$ in  $V(a,b,r_{1}' , r_{2}',n)$ if $r_{1}'$ and $r_{2}'$ are small enough. We have 
\[ d F^{-1} - \mathrm{id} = (dF)^{-1} \circ F^{-1} - \mathrm{id} \]
by the chain rule. 

We claim that $(dF)^{-1} - \mathrm{id} \in O_{\beta -1, p-1}$.  Observe that
\[ dF (z)^{-1} = (\det dF (z))^{-1} \mathrm{adj} (dF (z)). \]
Since the product of elements of $O_{\beta -1, p-1}$ belongs to $O_{\beta - 1, p-1}$, it is easy to see that 
\[ \det dF  -1 \in O_{\beta-1, p-1} \quad \mathrm{and} \quad \mathrm{adj} (dF) - \mathrm{id} \in O_{\beta -1, p-1} . \]
So in order to prove that $(dF)^{-1} - \mathrm{id} \in O_{\beta -1, p-1}$ we just need to show that given any $h_1, h_2 \in O_{\beta -1, p-1}$ we have
\[ \frac{h_1}{1+h_2} \in O_{\beta -1, p-1} \quad \mathrm{and} \quad \frac{1 + h_1}{1+h_2} -1  \in O_{\beta -1, p-1}, \]
which follows from  in Proposition \ref{clase11}. Now, we obtain 
\[ d F^{-1} - \mathrm{id} = (\mathrm{id} +  O_{\beta -1, p-1}) \circ F^{-1} - \mathrm{id} = (F^{-1} - \mathrm{id}) + O_{\beta -1, 0}  \in O_{\beta -1, 0}, \]
where the second equality follows from the identity $x \circ F^{-1} \equiv x$ and the inclusion of 
$F^{-1} - \mathrm{id}$ in $O_{\beta,0}$. 
In this way, we get $F^{-1} \in D_{\beta,1}$. 
In order to complete the proof, it is enough to show that 
$F^{-1} \in D_{\beta, q}$ for some $1 \le q < p$ implies $F^{-1} \in D_{\beta, q+1}$. So assume that  $F^{-1} \in D_{\beta, q}$ and $1 \le q < p$.
From Proposition \ref{pro:comp_O_D} we see that 
$h \circ F^{-1}\in O_{\beta -1, q}$ whenever $h \in O_{\beta-1,q}$.  Then
\[ d F^{-1} - \mathrm{id} = (\mathrm{id} +  O_{\beta -1, p-1}) \circ F^{-1} - \mathrm{id} = (F^{-1} - \mathrm{id}) + O_{\beta -1, q}  \in O_{\beta -1, q}, \]
whence $F^{-1} \in D_{\beta, q+1}$. 
\end{proof}

\section{Proof of Proposition \ref{smoothrealization}}\label{secsmooth}
This section is devoted to prove Proposition \ref{smoothrealization}. Since the proof is quite long, we organize it with the aid of several results
--- propositions  \ref{phiala}, \ref{fholder}, \ref{nabla} --- which will be proved along the way. 

Recall from the introduction to Section \ref{secmain} that  $\delta_0 (\nu +1) - \delta_1 \geq 4$.  We denote 
 \[ \phi(x,y)= y+\tilde{\phi}(x,y) \quad \mathrm{and} \quad \phi_{-1}(x,y)= y+\tilde{\phi}_{-1}(x,y) . \] 
We have 
\[ \tilde{\phi}  = O(|x|^{4} |y|^{\nu+1}) \quad \mathrm{and} \quad \tilde{\phi}_{-1}  = O(|x|^{4} |y|^{\nu+1}) \]
by Lemma \ref{expon}.
From now on, we denote 
$$\mathbb{D}^+= \left\{x\in\mathbb{C}\colon  |x|<1,\ \re(x)>0 \right\}$$
and $$\mathbb{D}^+(2)= \left\{x\in\mathbb{C}\colon |x|<2\  \re(x)>0 \right\}.$$
 \begin{pro} \label{phiala}
  The functions $\tilde{\phi}$ and $\tilde{\phi}_{\scriptscriptstyle -1}$ are holomorphic on
   $\widehat{\mathbb{D}}(2)\times B$ and satisfy the following properties:
  \begin{enumerate}
  \item\label{phial}    If restricted to 
 ${\mathbb{D}}^+(2)\times B$,  the  functions $\tilde{\phi}$ and $\tilde{\phi}_{\scriptscriptstyle -1}$  are of order $4$, in the sense of Definition \ref{definition}.
 \item \label{cota12}If $R>0$ is the radius of $B$, there exists $\mathfrak{r}\in (0, R)$ such that  
   \begin{align}\nonumber \left|d_y\tilde{\phi}(x,y)\right|, \left|d_y\tilde{\phi}_{-1}(x,y)\right|,
  \left|d_y\tilde{\phi}\big(x,\phi_{-1}(x,y)\big)\right|< \frac{1}{2}, \quad x\in\widehat{\mathbb{D}}(2),  |y|\le \mathfrak{r}.
 \end{align}  
 \end{enumerate} 
\end{pro}
\proof 
It is clear from the definitions that $\tilde{\phi}$ and $\tilde{\phi}_{\scriptscriptstyle -1}$ are holomorphic on
   $\widehat{\mathbb{D}}(2)\times B$. 
 Then, \eqref{phial} of Proposition \ref{phiala} follows from Proposition \ref{margot}
  for both  $\tilde{\phi}$ and  $\tilde{\phi}_{-1}$.
Let us prove \eqref{cota12} of Proposition \ref{phiala}. 
   From Lemma \ref{expon},
via the Cauchy inequality in the variable $y$, we find  $\tilde{c}_1>0$
such that 
\begin{align}\label{ruben} \left|d_y\tilde{\phi}(x,y)\right|\le \tilde{c}_1|x|^{4} |y|^{\nu}\quad \textrm{for }x\in\widehat{\mathbb{D}}(2) \textrm{ and }|y|<R/2.
\end{align}
 Then, for $\mathfrak{r}>0$ small enough,
 \begin{align}\label{fido}
|d_y\tilde{\phi}| 
< 1/2,  \quad  x\in\widehat{\mathbb{D}}(2),\; |y|\le \mathfrak{r},
\end{align} which is  the first inequality in \eqref{cota12} of Proposition \ref{phiala}. The second inequality in \eqref{cota12} is obtained exactly in the same way and we can assume this is done for the same 
$\mathfrak{r}$. Let us show the third inequality in  \eqref{cota12} of Proposition \ref{phiala}.
We have $\phi_{-1}(x,y) = O(y)$ by by Lemma \ref{expon}. 
  Therefore, if $y$ is small enough it follows  from \eqref{fido} that
 \begin{align}\nonumber
\left|d_y\tilde{\phi}\left(x,\phi_{-1}(x,y)\right)\right| 
< 1/2,  \quad  x\in\widehat{\mathbb{D}}(2).
\end{align} Finally, by reducing $\mathfrak{r}$ we can assume that this last inequality holds for
$x\in\widehat{\mathbb{D}}(2)$ and $|y|\le \mathfrak{r}$, which finishes the proof.
\qed

We start the construction of the diffeomorphism $F$. Let $B( \mathfrak{r})$ be denote the ball of radius $\mathfrak{r}$ centered at the origin of $\mathbb{C}^n$.
Take a smooth function $$\rho\colon \left[-\frac{\pi}{2}, \frac{\pi}{2}\right]\to[0,1]$$ such that
\begin{equation}\label{defiro}\rho(s)=\begin{cases}
 0 & \textrm{ if }\quad -\frac{\pi}{2}< s<-\frac{\pi}{4} \\
 1&\textrm{ if } \quad\frac{\pi}{4}< s<\frac{\pi}{2}\end{cases} \end{equation} 
and define the smooth map $\tilde{\sigma}_1\colon \Delta_1\times B( \mathfrak{r})\to \mathbb{C}^n$ by 
\begin{equation}\label{deco1}\tilde{\sigma}_1(x,y)=\begin{cases} (1-\rho(\arg x))\tilde{\phi}(x,y) & \textrm{ if }
\re(x)>0\\
0 &\textrm{ otherwise}\end{cases}\end{equation}
 Now, define the smooth functions \begin{align*} &\sigma_1\colon \Delta_1\times B( \mathfrak{r})\to \mathbb{C}^n 
 \\
 &F_1\colon \Delta_1\times B( \mathfrak{r})\to \mathbb{C}^{n+1} \\
 &\sigma_2\colon \Delta_2\times B( \mathfrak{r})\to \mathbb{C}^n \\
 &F_2\colon \Delta_2\times B( \mathfrak{r})\to \mathbb{C}^{n+1}\end{align*} by 
 
 \begin{align*}&\sigma_1(x,y)=y+ \tilde{\sigma}_1(x,y)\\
 &F_1(x,y)=(x,\sigma_1(x,y))\\  
 &\sigma_2(x,y)=\begin{cases}\sigma_1\left(x,\phi_{\scriptscriptstyle -1}(x,y)\right) &\textrm{if } \re(x)>0\\
y &\textrm{otherwise }
\end{cases}\\ 
&F_2(x,y)=(x,\sigma_2(x,y)).   \end{align*}
Let us make a remark about the domain of definition of ${\sigma}_2$: a priori this function is defined at the point $(x,y)$, $\re(x)>0$  if $\left(x,\phi_{\scriptscriptstyle -1}(x,y)\right)$ belongs to the domain of 
$\sigma_1$. However, it is easy to see that $\sigma_1$ can be defined at  $\left(x,\phi_{\scriptscriptstyle -1}(x,y)\right)$  whenever
$\phi$ is defined at  $\left(x,\phi_{\scriptscriptstyle -1}(x,y)\right)$. Thus, it follows from \eqref{phi-1} 
that $\sigma_1\left(x,\phi_{\scriptscriptstyle -1}(x,y)\right)$  can be defined for all $x\in{\mathbb{D}}$,
$\re (x)>0$,  $y\in B( \mathfrak{r})$ and therefore  ${\sigma}_2(x,y)$ is well defined on $ \Delta_2\times B( \mathfrak{r})$. \\
It follows directly from the definition \eqref{deco1} that 
\begin{equation}\nonumber\tilde{\sigma}_1(x,y)=\begin{cases} \tilde{\phi}(x,y) & \textrm{ if }
-\frac{\pi}{2}<\arg(x)<-\frac{\pi}{4}\\
0 &\textrm{ if } \frac{\pi}{4}<\arg(x)<\frac{3\pi}{2} \end{cases}\end{equation}
Therefore we have 
 \begin{equation}\label{deco1.1}F_1(x,y)=\begin{cases} (x,\phi(x,y)) & \textrm{ if }
-\frac{\pi}{2}<\arg(x)<-\frac{\pi}{4}\\
(x,y) &\textrm{ if } \frac{\pi}{4}<\arg(x)<\frac{3\pi}{2} \end{cases}\end{equation}
 The functions $F_1\circ \psi_1$ and $F_2\circ \psi_2$ --- where $\psi_1$ and $\psi_2$ are the charts of 
$\mathcal{M}$ --- are respectively defined on the sets 
$$\mathcal{M}'_1\colon= \psi_1^{-1}\left(\Delta_1\times B( \mathfrak{r} / 2 )\right)$$ and 
$$\mathcal{M}'_2\colon= \psi_2^{-1}\left(\Delta_2\times B( \mathfrak{r} / 2) \right).$$
The set $\mathcal{M}'\colon=\mathcal{M}'_1\cup \mathcal{M}'_2$ is clearly a neighborhood of the leaf $L$. By using \eqref{phi-1} and the expression of $\Psi=\psi_2\circ\psi_1^{-1}$ given in \eqref{cambio}, it is straightforward to check that  $F\colon \mathcal{M}'\to \mathbb{C}^{n+1}$ will be well defined by
\begin{align*} F=\begin{cases}F_1\circ \psi_1 &\textrm{ on }  \mathcal{M}'_1\\
F_2\circ\psi_2 &\textrm{ on } \mathcal{M}'_2\end{cases}
\end{align*}
Let us show that $F$ is a smooth diffeomorphism onto its image. Firstly we show that $F$ is injective.
Suppose that $F(p)=F(q)=(x_0,y_0)$. 
Since $\phi(x,y)=y+\tilde{\phi}(x,y)$, from \eqref{cota12} of Proposition \ref{phiala} we see that, { if }
$x\in\mathbb{D}^+$ and  $|y|< \mathfrak{r}$,
$$|d_y\phi|\le 1+|d_y\tilde{\phi}|<\frac{3}{2}.$$ Therefore,
if $|y|<\mathfrak{r}/2$ and $x\in {\mathbb{D}^+}$ we obtain
\[ |\phi(x,y)|=|\phi(x,y)-\phi(x,0)|< \frac{3}{2}|y|
<\frac{3\mathfrak{r}}{4}<\mathfrak{r}. \]
Therefore, if 
$x_0\in \Delta_1$ (resp. $x_0\in \Delta_2$) then $p, q \in  \psi_1^{-1}\left(\Delta_1\times B( \mathfrak{r}  )\right)$
(resp. $p, q \in  \psi_1^{-1}\left(\Delta_2\times B( \mathfrak{r}  )\right)$).
 Then it is enough to prove that $F_1$ is injective in $\psi_1^{-1}\left(\Delta_1\times B( \mathfrak{r}  )\right)$
 and, since ${\mathbb D} \setminus \Delta_1 = [0, -i)$, that $F_2$ is injective in a neighborhood of $\psi_2^{-1}\left((0,-i)\times B( \mathfrak{r}  )\right)$.
 The latter property is obvious since $F_2$ is the identity map in a neighborhood of $(0,-i)\times B( \mathfrak{r}  )$.
 So we focus on the injectivity of $F_1$.
  It follows from \eqref{cota12} of Proposition \ref{phiala} and \eqref{deco1} that  
 \begin{align}\label{cota13}\left|d_y{\tilde{\sigma}_1}(x,y)\right|, \left|d_y{\tilde{\sigma}_1}(x,\phi_{-1}(x,y))\right| <1/2, \quad x\in\Delta_1,\  y\in  B( \mathfrak{r}).\end{align}
Thus, for each $x\in \Delta_1$,  Lemma \ref{perturbacion} implies that the map  
$$y\in B(\mathfrak{r})\mapsto \sigma_1(x,y) \in\mathbb{C}^n$$ is a diffeomorphism onto its image. Then, for each  $x\in \Delta_1$, the function
 $F_1$ maps $\{x\}\times B(\mathfrak{r})$ diffeomorphically into the same fiber
$\{x\}\times\mathbb{C}^n$. This guarantees that $F_1$ is a diffeomorphism onto its image.   
Therefore $F$ is inyective and, since we have also proved that $F$ is a local diffeomorphism, we conclude that $F$ is a diffeomorphism onto its image. 
At this point,  \eqref{hojaL} and \eqref{fro} of Proposition \ref{smoothrealization} follows directly from the definition of $F$.
Clearly we have $F(\mathcal{M}')\subset \mathbb{C}^n\backslash E$. Let us show that 
$F(\mathcal{M}')$ contains the set $\mathbb{D}^*\times B( \mathfrak{r}/4)$, where $B( \mathfrak{r}/4)$ is the ball of radius $ \mathfrak{r}/4$ centered at $0\in\mathbb{C}^n$. Given $x\in \Delta_1$,    since $$|d_y\tilde{\sigma}_1(x,y)|<1/2,\quad y\in B(\mathfrak{r}),$$  by Lemma \ref{perturbacion}  the image of the map 
$$y\in B(\mathfrak{r} /2)\mapsto \sigma_1(x,y)\in \mathbb{C}^n$$
contains the ball $B( \mathfrak{r}/4)$. Thus, for each $x\in \Delta_1$
 the set $\{x\}\times B( \mathfrak{r}/4)$ 
is contained in the image of $F_1$, whence  $F_1 ({\mathcal M}_{1}')$ contains the set $\Delta_1\times B( \mathfrak{r}/4)$. 
Since 
\[ F_2 ((0,-i) \times B( \mathfrak{r}/2)) = (0,-i) \times B( \mathfrak{r}/2), \]
$F(\mathcal{M}')$ contains
  $\mathbb{D}^*\times B( \mathfrak{r}/4)$, so the set  $$U\colon=F(\mathcal{M}')\cup \big(\{0\}\times B( \mathfrak{r}/4)\big)$$
  is open.
 It is evident that $F(\mathcal{M}')=U\backslash E$ and therefore item \eqref{fro0.5} of Proposition 
 \ref{smoothrealization} holds by taking any  
 \begin{align}\label{rrr}r\le \mathfrak{r}/4.
 \end{align} 
 
 \begin{rem}\label{f11} If $x\in\mathbb{D}^+$ and $y\in B(\mathfrak{r})$, from \eqref{deco1} we
 have $$\tilde{\sigma}_1(x,y)=(1-\rho (\arg x))\tilde{\phi}(x,y).$$
 Thus, since $\tilde{\phi}$ is well defined on ${\mathbb{D}^+}(2)\times B(\mathfrak{r})$, the
 function $\tilde{\sigma}_1$ can be extended
to ${\mathbb{D}}^+(2)\times B(\mathfrak{r})$. 
 Therefore $F_1$ extends smoothly to 
 ${\mathbb{D}}^+(2)\times B(\mathfrak{r})$ as 
 $${F}_1(x,y)=(x,y+{\tilde{\sigma}}_1(x,y)).$$
 Now, exactly in the same way we obtained \eqref{cota13} we have
 \begin{align*}\label{cota13.1}\left|d_y{{\tilde{\sigma}}_1}(x,y)\right|, 
 \left|d_y{{\tilde{\sigma}}_1}(x,\phi_{-1}(x,y))\right| <1/2, \quad
  x\in{\mathbb{D}}^+(2),\ y\in B(\mathfrak{r}).
 \end{align*} From this, as we have done above  we prove that 
 ${F}_1$ is a diffeomorphism of ${\mathbb{D}}^+(2)\times  B(\mathfrak{r}) /2$ onto its image. 
\end{rem}
 
\begin{rem}\label{general} Exactly the same arguments used to prove that
 $$F_1(\{x\}\times B(\mathfrak{r}/2))\supset \{x\}\times B(\mathfrak{r}/4),\quad x\in\Delta_1$$
 work to prove that,   given a ball $B(\mathfrak{r}')$ of radius $\mathfrak{r}'\le \mathfrak{r}/2$ centered at $0\in\mathbb{C}^n$, we have
\begin{equation*} F_1(\{x\}\times B(\mathfrak{r}'))\supset \{x\}\times B(\mathfrak{r}'/2),\quad x\in\Delta_1.
\end{equation*}
\end{rem} 
\begin{pro}\label{fholder} 
 $F_1$ restricted to 
$$\nabla(\mathfrak{r}/2)\colon=\left\{(x,y)\in\mathbb{D}^*\times B\colon |\arg(x)|
<3 \pi/{4},\; |y|<\mathfrak{r}/2 \right\}$$  belongs to $D_4$.
\end{pro} 
\proof Recall that, if $(x,y)\in\mathbb{D}^+\times B(\mathfrak{r})$, 
$$F_1(x,y)=\left(x, y+[1-\rho(\arg x)]\tilde{\phi}\right),$$
so it is enough to show that  
$$\Big(\big[1-\rho(\arg x) \big]\tilde{\phi}\Big)\Big|_{\nabla(\mathfrak{r}/2)}$$  belongs to $O_4$, which in turn is consequence of Lemma \ref{expon}, Proposition \ref{margot}, Lemma \ref{argumento}
and Proposition \ref{clase11}. \qed

\subsection*{Construction of the vector fields $Y_1$ and $Y_2$.} 
The vector field $Z$ defines a holomorphic vector field ${\mathcal Z}$ in ${\mathcal M}'$ because $Z$ is invariant by the change of coordinates $\Psi$:
indeed,
 \[ (\Psi^{*} Z)  (x) \equiv Z(x \circ \Psi) \equiv Z(x) \equiv x  \]
implies $\Psi^{*} Z \equiv Z$. We define the vector fields 
 $$Y_1 \colon =(F)_*({\mathcal Z}) \quad \mathrm{and} \quad Y_2 \colon =(F)_*(i {\mathcal Z}) $$ 
in $F({\mathcal M}')$.
 The vector fields $Y_1$ and $Y_2$ are tangent to the foliation  $\mathcal{F}^{\scriptscriptstyle U}$  --- as defined in \eqref{fro11} of Proposition \ref{smoothrealization}. 
 From Remark \ref{general}, we have 
 $$F_1(\Delta_1\times B(\mathfrak{r}/2))\supset \Delta_1\times B(\mathfrak{r}/4),$$ 
 so  since $F_2((0,-i)\times B(\mathfrak{r}/4)) =  (0,-i)\times B(\mathfrak{r}/4)$, $Y_1$ and $Y_2$ are defined on the set 
 $${\mathbb D}^{*} \times B(\mathfrak{r}/4).$$

\begin{pro}\label{nabla}The vector fields $Y_1 - Z$ and $Y_2 - i Z$ restricted to 
$$\Big\{(x,y)\in\mathbb{D}^*\times B\colon |y|< {\mathfrak r}/4 \Big\}.$$  
belong to $O_3$. In particular, $Y_1$ and $Y_2$ extend to $U(r)$ (see Proposition
\ref{smoothrealization}) as $C^{2}$ vector fields for $r>0$ small enough. 
\end{pro}
\begin{proof} 
We have $F_1 \in D_4$ by Proposition \ref{fholder} and thus $dF_1 - \mathrm{id} \in O_3$.
Moreover $F_{1}^{-1} \in D_4$ by Proposition \ref{pro:inverse}.
It follows from Proposition \ref{clase11} that
\[ dF_1 \cdot Z  - Z \in O_3. \]
 We have
$$Y_1 =(dF_1\cdot Z)\circ F_1^{-1} \implies Y_1 -  Z  = (Z \circ F_{1}^{-1} - Z) + [dF_1 \cdot Z - Z] \circ F_{1}^{-1}. $$
The second term on the right hand side of last equality belongs to $O_3$ by Proposition \ref{pro:comp_O_D}
whereas $Z \circ F_{1}^{-1} - Z \in O_4$ by Proposition \ref{pro:discrete_derivative}.
Therefore $Y_1 - Z$ belongs to $O_3$ for the domain $\nabla(\mathfrak{r}/4)$.
Thus,
since $Y_{1} - Z \equiv 0$ on 
$$\Big\{(x,y)\in\mathbb{D}^*\times B\colon \pi/4 < \arg(x) <7\pi/4, \; |y|< {\mathfrak r}/4 \Big\},$$ 
we get $Y_{1} - Z  \in O_3$ on  ${\mathbb D}^*\times  B({\mathfrak r}/4)$.
Clearly $Y_1$ extends to 
$\{x=0\} \times B(r)$, for some $r>0$, as a $C^{2}$ vector field such that  $(Y_{1})|_{x=0} \equiv Z |_{x=0}$ by 
Proposition \ref{pro:ck_extension}. 
The arguments for $Y_2$ are analogous and we get $(Y_{2})|_{x=0} \equiv i Z |_{x=0}$.
\end{proof} 

\subsection*{End of the Proof of Proposition \ref{smoothrealization}}
By Proposition \ref{nabla}   we have that $Y_1$ and $Y_2$ extend to 
$$U(r)=\mathbb{D}\times B(r)$$ in the class $C^{2}$.  Therefore \eqref{fro0.8} of
Proposition \ref{smoothrealization} is proved.
Assertion  
\eqref{fro1} of Proposition \ref{smoothrealization} follows directly from the definition of $Y_1$ and $Y_2$. 
The vector fields $Y_1$ and $Y_2$ are real-linearly independent on ${\mathbb D}^{*} \times B(r)$ by construction. 
Since 
\begin{equation}
\label{equ:aux_ext}
Y_1 \equiv Z \quad \mathrm{and} \quad Y_{2} \equiv iZ \;\; \mathrm{ on } \;\; \{ 0 \} \times B(r),  
\end{equation} 
they are linearly independent on $U(r) \setminus \{0\}$,
whence we obtain  \eqref{fro2} of Proposition \ref{smoothrealization}.
Moreover, evidently \eqref{equ:aux_ext}  also proves \eqref{fro2.1} of Proposition \ref{smoothrealization}.
Since $Y_1$ is of class $C^2$ and all derivatives of order up to $2$ of $Y_1 - Z$ vanish on $\{x=0\}$ by Proposition \ref{pro:ck_extension}, 
we conclude that 
$$dY_1(0)=dZ(0),$$ so \eqref{fro3} of Proposition \ref{smoothrealization} is proved.\qed

\section{Extension of the almost complex structure}\label{secextension}
We first recall some basic facts about almost complex structures. Let $\mathscr{U}\subset\mathbb{R}^{2m}$ be an open set. Let $u_1,\dots, u_m$ be vector fields on 
$\mathscr{U}$ that are real linearly independent at every point of $\mathscr{U}$. An almost complex structure $\mathscr{J}$ on $\mathscr{U}$ is determined by  the vector fields
$$v_1=\mathscr{J}(u_1),\dots, v_m=\mathscr{J}(u_m).$$ 
An arbitrary choice of the vector fields  $v_1,\dots,v_m$ will indeed define an almost complex structure on $\mathscr{U}$ if the vector fields
$$u_1,\dots,u_m,v_1,\dots,v_m$$
define a real frame on $U$: in this case it is easy to see that 
$$\mathscr{J}(u_1)=v_1,\dots, \mathscr{J}(u_m)=v_m, \ \mathscr{J}(v_1)=-u_1,\dots, \mathscr{J}(v_m)=-u_m$$
actually define an almost complex structure. 

Now, let $U\subset\mathbb{C}^{n+1}$ as in Proposition \ref{smoothrealization}.  Consider in $U\backslash E$  the vector fields $$u=\frac{\partial}{\partial x}, 
\quad u_j=\frac{\partial}{\partial y_j}, \quad j=1,\dots, n,$$
regarded as real vector fields.
The almost complex structure $J$ on $U\backslash E$ is determined by the vector fields
$$J(u),\quad J(u_j),\quad j=1,\dots,n.$$
In the chart $\psi_1=(x,y)$ the map $F$ is given by  $F_1$, so the structure $J$ on the image of $F_1$ is the 
pushforward by  $F_1$ of the canonical complex structure.  Recall that, if
$a\in\Delta_1$, then $F_1$ maps the vertical $\{x=a\}$ biholomorphically into itself. Therefore $F_1$ preserves the canonical complex structure of $\displaystyle{\{x=a\}}$. This means that the restriction of $J$ to ${\{x=a\}}$ is just the canonical almost
complex structure of $\displaystyle{\{x=a\}}$. Then, since the vector fields $u_j$ are tangent to $\{x=a\}$, we have
that  $J(u_j)=iu_j $ on $\{x=a\}$. After doing 
the same argument with the other chart $\psi_2$ we conclude that
\begin{align}\label{fini}J(u_j)=iu_j \textrm{ on } U\backslash E, \quad j=1,\dots,n. 
\end{align}
Since the vector fields $u,u_1,J(u_1),\dots,u_n,J(u_n)$ extend smoothly to $\mathbb{C}^{n+1}$,
the  extension of $J$ to $U$ is equivalent to the extension of the vector field $J(u)$ to $U$ in such a way
the vector fields $$u,J(u),u_1,J(u_1),\dots,u_n,J(u_n)$$ define a real frame on $U$. 
Since $J$ is the pushforward by $F$ of the complex structure on $\mathcal{M}$, we have that $J(p)$, 
$p\in U$ is the standard
almost complex structure whenever $F$ has a complex linear derivative at the point  
$F^{-1}(p)$. Thus, in view of \eqref{fro} of Proposition \ref{smoothrealization},  the structure $J$ 
coincides with the standard one in the complement in $U\backslash E$ of the set  
  $$ \nabla\colon =\left\{(x,y)\in U\colon   -\frac{\pi}{4}\le\arg(x) \le\frac{\pi}{4}\right\}. $$
  In particular, we have 
 \begin{equation}\label{holderi}J(u)=iu\quad \textrm{on}\quad  U(r)\backslash \overline{\nabla}\subset U\backslash E .\end{equation}

\begin{pro}\label{c1e} $J(u)$  satisfies $J(u) -iu \in O_3$  and extends  to $U(r)$ in the class $C^{2}$ for $r>0$ small enough.  
\end{pro}
\begin{proof}  
 If  $p\in \nabla(r)$ (see Proposition \ref{fholder}),
we have 
\begin{align}\label{j(u)}
  J(u)(p)&= dF_1(F_1^{-1}(p))\cdot idF_1^{-1}(p)\cdot u.
 \end{align} 
 The diffeomorphism $F_1$ belongs to $D_4$ (Proposition \ref{fholder}) and so does $F_{1}^{-1}$ by Proposition \ref{pro:inverse}.
 Therefore, we get
 \[ idF_1^{-1}\cdot u  - i u  \in O_3 .\]
Notice that 
\[ dF_1  \circ F_{1}^{-1} - \mathrm{id} = (dF_1 - \mathrm{id})  \circ F_{1}^{-1} + (F_{1}^{-1} -  \mathrm{id}).  \]
The first term in the right hand side belongs to $O_3$ by Proposition \ref{pro:comp_O_D} whereas $F_{1}^{-1} \in D_4$ implies that the second 
term belongs to $O_4$. Then
\[ dF_1  \circ F_{1}^{-1} - \mathrm{id} \in O_3 . \]
 Therefore, since 
 \[ J(u) - iu = (dF_1  \circ F_{1}^{-1} - \mathrm{id} ) \cdot idF_1^{-1}\cdot u +(idF_1^{-1}\cdot u  - i u ), \]
 we see that $J(u) - iu \in O_3$ on $\nabla(r)$ for $r>0$ small enough. Moreover $J(u) \equiv iu$ in $U(r) \setminus \overline{\nabla}$, so it is easy to see that
 $J(u) -i u \in O_3$ for $U(r) \setminus E$. It follows from Proposition \ref{pro:ck_extension} that $J(u)$ extends in the class $C^{2}$ to $U(r)$ for $r>0$ small enough and 
 $J(u)|_{E} \equiv (iu)|_{E}$ by Proposition \ref{pro:ck_extension}.
\end{proof}

\subsection*{Proof of Proposition \ref{jexten}}  
 By Proposition \ref{c1e} the vector field $J(u)$ extends to $U(r)$ in the class $C^{2}$. Then the vector fields $$u,J(u),u_1,J(u_1),\dots,u_n,J(u_n)$$ are 
of class $C^{2}$ on $U(r)$.  Thus, the first assertion of Proposition \ref{jexten} reduces to showing that the vector fields above define a real frame on $U(r)$. Since 
they do so on $U(r)\backslash E$, it remains to look at the points in $E\cap U(r)$, whence
the proof of the proposition reduces to the proof of its last assertion.  Let  $p\in E\cap U(r)$. From \eqref{fini}
we obtain that
 \begin{align}\nonumber J(u_j)(p)=iu_j (p), \quad j=1,\dots,n. 
\end{align} Then it suffices to prove that $J(u)(p)=iu(p),$ which is true because 
the equality  \eqref{holderi} extends to $E\cap U(r)$ by continuity.
\qed

\subsection*{Proof of Proposition \ref{new}}
 Since the structure $J$ on $U(r)\backslash E$ comes from the complex structure of $\mathcal{M}$, we have that 
$J|_{(U(r)\backslash E)}$ satisfies the integrability condition for almost complex structures. Thus, by continuity, $J$ satisfies the integrability condition on the whole $U(r)$ and,  by a direct application of Theorem \ref{teonw}, we obtain a local diffeomorphism ${G}_1$ of class $C^{2,1/(2n+2)}$ between neighborhoods of $0\in\mathbb{C}^{n+1}$, which takes the structure $J$ to the standard complex structure of  $\mathbb{C}^{n+1}$. 
The vector field $u$ is obviously tangent to the submanifold $\{y=0\}$. Since  $\{y=0\}$ is pointwise
fixed by $F_1$, from \eqref{j(u)} and  \eqref{holderi} we see that  $J(u)=i u$ on  $\{y=0\}$, which means
that $J(u)$ 
 is also tangent to the submanifold $\{y=0\}$. Therefore
$\{y=0\}$ is an almost complex submanifold of $(U(r),J)$. On the other hand, as we have seen in the beginning of this section, the restriction of $J$ to a
fiber $\{x=a\}$, $a\neq 0$ 
coincides with the canonical almost complex structure on $\{x=a\}$. Clearly  this fact extends to $E=\{x=0\}$ by continuity, which allow us to
   conclude that $E$ is also an almost complex submanifold of $(U(r),J)$.    Thus,  the images of
$\{y=0\}$ and $E$  by ${G}_1$ become 
complex submanifolds in $\mathbb{C}^{n+1}$, they are transverse to each other at the origin, one
of them has dimension one and the other has codimension one. Then we can find a local biholomorphism
 $$G_2\colon (\mathbb{C}^{n+1},0)\to (\mathbb{C}^{n+1},0)$$ taking these submanifolds to $\{y=0\}$  and $E$. Thus, if we choose $G=G_2\circ {G_1}$, take $r'>0$  small enough and $U'$ as in the statement of Proposition \ref{new}, and set $\Omega=G(U')$, 
 the assertions  \eqref{new1}, \eqref{new2}  and \eqref{new3} of Proposition \ref{new} hold.  
 Since $dG$ takes the   structure $J$ to the canonical complex structure, in particular we have that
 $$dG(0)\circ J(0)=idG(0).$$ Thus,  since Proposition \ref{jexten} assures that $J(0)$ is just the multiplication by $i$, we find that $dG(0)$ commutes with the canonical complex structure, which
 means that $$dG(0)\colon \mathbb{C}^{n+1}\to\mathbb{C}^{n+1}$$ is a complex linear isomorphism. 
 Since $G$ leave $\{y=0\}$ and $\{x=0\}$ invariant, the same holds for $dG(0)$.  Thus, 
 we take $dG(0)^{-1}\circ G$ instead of $G$, whence  \eqref{new1}, \eqref{new2}  and \eqref{new3} of Proposition \ref{new}  remain true, but now we have  $$dG(0)=\id,$$  
 so \eqref{new4} of Proposition \ref{new}  holds.
 Let $\mathcal{F}$ be the pushforward by $G$ of the foliation
  $\mathcal{F}^{\scriptscriptstyle U}|_{U' \backslash\{0\}}$.  Then  $\mathcal{F}$ is a $C^1$ foliation on
  $\Omega\backslash \{0\}$.  
  Recall that 
   the almost complex structure of $\mathcal{M}$ is taken by $F$ to the almost complex structure 
   $J$ on $U\backslash E$, which in turns is taken by $G$ to the canonical complex structure of 
   $\mathbb{C}^{n+1}$. This means that  $G\circ F$ maps 
   $$\tilde{\mathcal{M}}\colon=F^{-1}(U' \backslash E)$$ biholomorphically onto $\Omega\backslash E$ and takes
   the foliation 
  $\mathcal{F}^{\scriptscriptstyle \mathcal{M}}$ to the foliation $\mathcal{F}|_{\Omega\backslash E}$,
  whence $\mathcal{F}$ is holomorphic on $\Omega\backslash E$. It follows from Proposition
  \ref{extefol} that $\mathcal{F}$ is holomorphic on $\Omega\backslash \{0\}$, so  \eqref{new5} of Proposition \ref{new} holds. 
   Let $Y$ be the pushforward of $Y_1$ by the diffeomorphism $G$, that is
  $$Y(p)= dG(G^{-1}(p))\cdot Y_1(G^{-1}(p)),\quad p\in \Omega.$$ 
  Notice that $Y$ coincides with the pushforward of ${\mathcal Z}$ by $G \circ F$ in $\Omega \setminus E$ and hence is holomorphic
  in $\Omega \setminus E$. Since $Y$ is continuous, it is holomorphic in $\Omega$ and we conclude that $Y$ is a holomorphic vector
  field generating ${\mathcal F}$ with isolated singularity.
   Moreover, since $dG(0)=\id$ and
  $dY_1(0)= d Z (0)$   --- see Proposition \ref{smoothrealization}, we obtain 
  \[ dY(0)=dZ(0) \] 
  and thus assertion \eqref{new6} of Proposition \ref{new} follows. \\
  By   \eqref{hojaL} of Proposition  \ref{smoothrealization},
    $$\mathbb{D}^*\times\{0\}=F(L)$$ 
  is a leaf of $\mathcal{F}^{\scriptscriptstyle U}$.  Then
  $$ F(L) \cap U' = \{y=0\}\cap U'$$ 
  is a leaf of   $\mathcal{F}^{\scriptscriptstyle U}|_{U'}$.
  Therefore  
  $$S\colon=G( \{y=0\}\cap U' )=(\{y=0\}\cap \Omega)\backslash \{0\} $$ 
  is a leaf of $\mathcal{F}$ --- that is, $\{y=0\}$ is a separatrix of $\mathcal{F}$ ---  and we have that
  $$\tilde{L}\colon= (G\circ F)^{-1}(S)\subset L$$
  is a leaf of $\mathcal{F}^{\scriptscriptstyle \mathcal{M}}$ restricted to 
  $$\tilde{\mathcal{M}}=F^{-1}(U' \backslash E).$$
  Thus, $G\circ F$ maps 
   $\tilde{\mathcal{M}}$  onto $\Omega\backslash E$ and maps the leaf $\tilde{L}$ of 
   $\mathcal{F}^{\scriptscriptstyle \mathcal{M}}|_{\tilde{\mathcal{M}}}$ onto the leaf $S$ of $\mathcal{F}$.
   Since  $G\circ F$ s a holomorphic equivalence between the foliations 
   $\mathcal{F}^{\scriptscriptstyle \mathcal{M}}|_{\tilde{\mathcal{M}}}$ and $\mathcal{F}|_{\Omega\backslash E}$, 
   the holonomy of $S$ is holomorphically conjugated to that of $\tilde{L}$, which in turn --- by Proposition \ref{hojal} ---  is holomorphically conjugated to $\mathfrak{h}$. 
  Proposition \ref{new}  is proved. \qed
 
 \section{An aplication of the realization theorem} \label{application}
 We will use Theorem   \ref{maintro1} to construct an example of a germ of vector field in $(\mathbb{C}^{3},0)$
that has closed leaves except in a hypersurface.
 The following example of local diffeomorphism in $(\mathbb{C}^{2},0)$ was introduced in the thesis of L. Lisboa \cite{Lisboa-Ribon:finite}:
 given $\lambda \in {\mathbb S}^{1}$ that is not a root of unity and such that $\liminf_{j \to \infty} \sqrt[j]{|\lambda^{j} -1|} = 0$, i.e. a Cremer number, 
 there exist $f \in \diff (\mathbb{C}^{2},0)$ of the form 
 \[ f(y_1,y_2) = (\lambda y_1, y_2 + a (y_1)), \]
 where $a(y_1) \in {\mathbb C}\{y_1\}$,
 and an open neighborhood $U$ of the origin such that any orbit of the restriction of $f$ to $U$ is finite.
 We remark that
 the positive orbit of a point $y$ in $U$ is the set
 \[ \{ f^{j} (y)\colon  j \in {\mathbb Z}_{\geq 0}, \ f^{k} (y) \in U \ \forall k=0,\dots, j \} ;\]
 the negative orbit is defined analogously. We say that the orbit of $y$  is finite if both the positive and negative orbits of  $y $
 are finite sets. There are two types of finite orbits, namely periodic orbits --- those such that $f^{k} (y) =y$ for some $k \in {\mathbb Z}^{*}$ ---
 and the orbits such that $f^{k} (y)$ is defined in $U$ for finitely many $k \in {\mathbb Z}$.

  It is easy to see that $f$ is formally conjugated to $(y_1, y_2) \mapsto (\lambda y_1, y_2)$ by a transformation of the form 
 $(y_1, y_2) \mapsto (y_1, y_2 - b(y_1))$. Moreover, $f$ has exactly one analytic invariant curve and one divergent formal
 invariant curve, namely $y_1 = 0$ and $y_2 = b(y_1)$, respectively.

 Denote $A = \mathrm{diag} (\mu_1, \mu_2)$ where we 
 choose $\mu_1 \in {\mathbb R}^{-} \setminus {\mathbb Q}$ such that $e^{2 \pi i \mu_1} = \lambda$ and $\mu_{2} \in {\mathbb Z}^{-}$.
 Since $f$ is formally linearizable we can apply Theorem \ref{maintro1}
 to obtain a vector field $Y$  of the form  \eqref{sistema0} whose holonomy is analytically conjugated to $f$. 
 The analytic invariant curve of $f$ generates a smooth invariant hypersurface for $Y$. Similarly, 
it can be 
 seen that there exists a divergent formal invariant hypersurface generated by the purely formal invariant curve of $f$.
 The restriction of $Y$    to $E=\{x=0\}$ is a system of the form 
 \[ y' = A \cdot y + G(0,y). \]
 Since $(\mu_1, \mu_2)$ is in the Poincar\'{e} domain and such eigenvalues have no resonances, $Y|_{E}$ is analytically conjugated to 
 $y' = A \cdot y$ by a tangent to the identity map $h \in \diff (\mathbb{C}^{2},0)$, so we can assume $G(0,y) \equiv 0$. Thus, if $\mathcal{F}$ is the foliation generated by $Y$ on a small enough neighborhood
 $W$ of the origin,  We have that 
 \begin{itemize}
 \item the axes  $x$, $y_1$ and $y_2$ are the separatrices of $\mathcal F$;
 \item the leaves of ${\mathcal F}$ in $E$ other than the separatrices are non-closed in $W$. 
 \end{itemize}
 Moreover, provided $W$ is small enough and using the fact that $f$ has finite orbits in some neighborhood of the origin, it is not difficult to prove that all the leaves outside $E$ and different from the $x$ axis are closed in $W$. 
 \begin{rem} Regarding a local one dimensional foliation, we say that a leaf is closed if its closure is a closed analytic curve. This definition allows us to consider the separatrices as closed leaves. Thus, the foliation ${\mathcal F}$ above is an example in which all non-closed leaves
 are concealed in a proper analytic set, namely the divisor $E$.  
 The example makes clear that if we have a one-dimensional foliation with 
 closed leaves and we consider a birational transformation, for instance a blow-up of a point, we can not assume a priori
 that the transform of the foliation still has closed leaves. 
 \end{rem}

\end{document}